\documentclass[a4paper,leqno,12pt]{amsart}
\usepackage{amssymb,amscd}
\usepackage[latin1]{inputenc}
\usepackage{graphicx}
\usepackage{fullpage}
 \divide\oddsidemargin by 2
\usepackage{enumerate}
\makeatletter
\let\@listlla\list

\def\list#1#2{\@listlla{#1}{#2\itemsep=2pt\parsep=0pt\topsep=3pt plus 1pt minus 1 pt}}
\newcommand{\tridiagram}[6]{{\par\par \centering
\@picture(120,120)(0,0) \put(30,95){\makebox(0,0)[r]{$#1$}}
\put(90,30){\makebox(0,0)[tl]{$#3$}}
\put(90,95){\makebox(0,0)[l]{$#2$}}
\put(60,102){\makebox(0,0)[b]{$#4$}}
\put(102,60){\makebox(0,0)[l]{$#6$}}
\put(50,50){\makebox(0,0)[tr]{$#5$}} \thinlines
\put(40,95){\vector(1,0){40}} \put(95,80){\vector(0,-1){40}}
\put(25,80){\vector(1,-1){55}}
\endpicture\par\par}\noindent\ignorespaces}
\def\@map#1#2[#3]{\mbox{$#1 \colon #2 \longrightarrow #3$}}
\def\map#1#2{\@ifnextchar [{\@map{#1}{#2}}{\@map{#1}{#2}[#2]}}
\newcommand{\Aa}{{\mathcal A}}

\newcommand{\Ff}{{\mathcal F}}

\newcommand{\Jj}{{\mathcal J}}

\newcommand{\Oo}{{\mathcal O}}

\newcommand{\PP}{{\bf P}}

\newcommand{\Aut}[1]{\mbox{\rm Aut}{(#1)}}
\newcommand{\Out}[1]{\mbox{\rm Out}{(#1)}}
\newcommand{\Inn}[1]{\mbox{\rm Inn}{(#1)}}

\newcommand{\mr}{\mathbb{R}}

\newcommand{\mc}{\mathbb{C}}
\newcommand{\mz}{\mathbb{Z}}
\newcommand{\mh}{\mathbb{H}}

\newcommand{\mn}{\mathbb{N}}

\newcommand{\md}{\mathbb{D}}

\newcommand{\Z}{\ensuremath{\mathbb{Z}}}

\newcommand{\SI}{\ensuremath{\mathbb{S}^1}}
\newcommand{\F}[1]{\ensuremath{\mathbb{F}_{#1}}}

\makeatother
\newtheorem{theorem}{Theorem}[section]
\newtheorem{theo}{Theorem}[]

\newtheorem{proposition}[theorem]{Proposition}
\newtheorem{corollary}[theorem]{Corollary}
\newtheorem{lemma}[theorem]{Lemma}
\theoremstyle{definition}
\newtheorem{definition}[theorem]{Definition}

\newtheorem{remark}[theorem]{Remark}

\usepackage{color}

\title{Poisson boundary of groups acting on $\mr$-trees \footnote{Version January 2011}}
\author{Fran\c{c}ois Gautero, Fr\'ed\'eric Math\'eus}
\address{Fran\c{c}ois Gautero, Universit\'e de Nice Sophia Antipolis,
Laboratoire de Math\'ematiques J.A. Dieudonn\'e (U.M.R. C.N.R.S. 6621), Parc Valrose, 06108
Nice, France} \email{Francois.Gautero@unice.fr}
\address{Fr\'ed\'eric Math\'eus, Universit\'e de Bretagne-Sud, L.M.A.M., Campus de Tohannic
BP 573, 56017 Vannes, France} \email{Frederic.Matheus@univ-ubs.fr}


\begin{document}

\begin{abstract}
We give a geometric description of the Poisson boundaries of certain
extensions of free and hyperbolic groups. In particular, we get a
full description of the Poisson boundaries of free-by-cyclic groups
in terms of the boundary of the free subgroup.
We rely upon the description of Poisson boundaries by means of a
topological compactification as developed by Kaimanovich. All the
groups studied here share the property of admitting a sufficiently
complicated action on some $\mr$-tree.
\end{abstract}

\maketitle

\section*{Introduction}

Let $G$ be a finitely generated group and $\mu$ a probability
measure on $G$. A bounded function $f:G\rightarrow \mr$ is {\em
$\mu$-harmonic} if
\[
\forall g\in G\: , \: f(g)\:=\:\sum_{h}\mu(h)f(gh)\:.
\]
Denote by $H^{\infty}(G,\mu)$ the Banach space of all bounded
$\mu$-harmonic functions on $G$ equipped with the sup norm. 
There is a simple way to build such functions. Consider a probability 
$G$-space $(X,\lambda)$ such that the measure  $\lambda$ is
$\mu$-stationary, which means that 
$$\lambda\:=\:\mu*\lambda\:\buildrel \rm def \over = \sum\mu(g)\cdot g\lambda\:\:$$
To any function $F\in L^{\infty}(X,\lambda)$ one
can assign a function $f$ on $G$ defined, for $g\in G$, by the {\em
Poisson formula} $f(g)\:=\:\langle F,g\lambda\rangle\:=\:\int_{X}F(gx)d\lambda(x)$.
This function $f$ belongs to $H^{\infty}(G,\mu)$ since $\lambda$ is
$\mu$-stationary.
In \cite{Fu63,Fu71}, Furstenberg constructed a measured
$G$-space $\Gamma$ with a $\mu$-stationary measure $\nu$ such that
the Poisson formula states an isometric isomorphism between
$L^{\infty}(\Gamma,\nu)$ and $H^{\infty}(G,\mu)$. The space
$(\Gamma,\nu)$ is called the {\em Poisson boundary} of the measured
group $(G,\mu)$. It is well defined up to $G$-equivariant measurable
isomorphism. A natural question is to decide whether a given $G$-space $X$ 
endowed with a certain $\mu$-stationary measure $\lambda$ is actually the 
Poisson boundary of the pair $(G,\mu)$.

\medskip

Let $F$ be a finitely generated free group, and $G$ a finitely generated
group containing $F$ as a normal subgroup. The action of $G$ on $F$ by inner 
automorphisms extends to an action by homeomorphisms on the geometric boundary $\partial F$
of $F$. The group $G$ is said to be a cyclic extension of $F$
if the factor $G/F$ is an infinite cyclic group. A probability measure $\mu$ on
$G$ is non-degenerate if the semigroup generated by its support is $G$. 
In this paper, we prove:

\begin{theo}\label{short}
Let $G$ be a cyclic extension of a free group $F$, and $\mu$ a 
non-degenerate probability measure on $G$ with a finite first moment
$($with respect to some gauge$)$. Then there exists a unique $\mu$-stationary probability 
measure $\lambda$ on $\partial F$ and the $G$-space $(\partial F,\lambda)$ is the 
Poisson boundary of $(G,\mu)$.
\end{theo}

The Poisson boundary is trivial for all non-degenerate measures on abelian groups
\cite{Cho-De} and nilpotent groups \cite{Dyn-Ma} : on such groups, there are no 
non constant bounded harmonic functions. If $F$ is a finitely generated free group
then the action mentioned above on $\partial F$ is just the left-action. Let $\mu$ be
a non-degenerate probability measure on $F$. Then there exists a unique $\mu$-stationary 
probability measure $\lambda$ on $\partial F$ \cite{ledr}. If $\mu$ has a finite first
moment, then $(\partial F,\lambda)$ is the Poisson boundary of $(F,\mu)$ 
\cite{Dyn-Ma,Derrien2,Kaim4}. 

\medskip

Assume now that $G$ contains $F$ as a normal subgroup. Then  there still exists a unique 
probability measure $\lambda$ on $\partial F$ which is $\mu$-stationary for the aforementioned 
action of $G$ on $\partial F$, and the $G$-space $(\partial F,\lambda)$ is a so-called
$\mu$-boundary of $G$, that is a $G$-equivariant quotient of the Poisson boundary of $(G,\mu)$ \cite{Vershik}. 
So the question is to find conditions on $G$ and $\mu$ ensuring that $(\partial F,\lambda)$ {\it is}
actually the whole Poisson boundary of $(G,\mu)$. For instance, if $F$ has finite index in $G$, then
there exists a measure $\mu^{0}$ on $F$ such that the Poisson boundaries of $(G,\mu)$
and $(F,\mu^{0})$ are isomorphic \cite {Fu71}. Moreover, if $\mu$ has a finite first moment, then 
so has $\mu^{0}$ \cite{Kaim2}, hence the Poisson boundary of $(G,\mu)$ is $(\partial F,\lambda)$.

\medskip

Theorem \ref{short} above provides an answer when $G$ is a cyclic extension of the free
group $F$. In this case, there is an exact short sequence 
$\{1\} \rightarrow F \rightarrow G \rightarrow \mz \rightarrow \{0\} $.
Such an extension is obviously splitting: there exists an automorphism $\alpha$ of 
$F$ such that $G$ is isomorphic to the semi-direct product $F \rtimes_{\alpha} \mz$.
Of course, the Poisson boundary is already known if $\alpha$ is an inner automorphism of $F$, 
because $G$ is then isomorphic to a direct product $F \times \mz$. Therefore, since $\mz$ is 
central, its acts trivially on the Poisson boundary of $G$ which can be identified to the 
Poisson boundary of $G/\mz \sim F$ \cite{Kaim5}.

\medskip

There are two types of growth for an automorphism of a finitely generated group:
polynomial growth, and exponential growth. It refers to the growth of the iterates
of a given element under the automorphism, see Def. \ref{growth} below. In the case of
an automorphism $\alpha$ of a finitely generated free group $F$, by the pioneering work of Bestvina-Handel \cite{BestvinaHandel} (see
also \cite{Levitt}  for precise and detailed statements about the growth of free group automorphisms) there is a dichotomy:
\begin{itemize}
   \item [$\bullet$] either $\alpha$ has a polynomial growth, $i.e.$ $|\alpha^n(g)|$ grows
                   polynomially for every $g\in F$,
   \item [$\bullet$] or $\alpha$ has an exponential growth, $i.e.$ there exists $g\in F$ such that
                   $|\alpha^n(g)|$ grows exponentially.
\end{itemize}

If the automorphism $\alpha$ has exponential growth, then the group $G = F \rtimes_{\alpha} \mz$
is hyperbolic relatively to the family ${\mathcal H}$ of polynomially growing subgroups \cite{Gautero1}, 
therefore its Poisson boundary may be identified with its relative hyperbolic boundary 
$\partial^{RH}(G,{\mathcal H})$ (see Section \ref{rh}). But this answer is not completely satisfactory 
for us since the link between $\partial F$ and the Poisson boundary of $G$ remains unclear.

\medskip

We list now a few generalizations of Theorem 1. First, we can deal with non-cyclic
extension, but we have a restriction on the growth. In the next statement, we denote
by $\Out{F} = \Aut{F} / \Inn{F}$ the group of outer automorphisms of $F$ where $\Inn{F}$
stands for the group of inner automorphisms.

\begin{theo}\label{polynomial}
Let $G$ be a semi-direct product $G = F \rtimes_\theta \mathcal P$ of a free group $F$
with a finitely generated subgroup $\mathcal P$ over a monomorphism 
$\theta \colon {\mathcal P} \rightarrow \Aut{F}$ such that $\theta({\mathcal P})$ maps 
injectively into $\Out{F}$ and consists entirely of polynomially growing automorphisms.
Let $\mu$ be a non-degenerate probability measure on $G$ with a finite first moment. 
Then there exists a unique $\mu$-stationary probability 
measure $\lambda$ on $\partial F$ and the $G$-space $(\partial F,\lambda)$ is the Poisson 
boundary of $(G,\mu)$.
\end{theo}

If $\theta({\mathcal P})$ is trivial in $\Out{F}$ then $G$ is isomorphic to a direct
product $F \times F'$ where $F'$ is a finitely generated free group. Then modifying the proof
of Theorem \ref{polynomial} yields:

\begin{theo}\label{direct} 
Let $G = F \times F'$ be the direct product of two finitely generated free groups. 
Consider the left-action of $G$ on $\partial F$ given by some fixed monomorphism 
$F' \hookrightarrow \Inn{F}$. Let $\mu$ be a non-degenerate probability measure on $G$ with a finite 
first moment. Then there exists a unique $\mu$-stationary probability measure $\lambda$ on 
$\partial F$ and the $G$-space $(\partial F,\lambda)$ is the Poisson boundary of $(G,\mu)$.
\end{theo}

Theorems \ref{polynomial} and \ref{direct} describe the Poisson boundary of all the groups 
$G = F \rtimes_\theta \mathcal P$ such that $\theta({\mathcal P})$  consists entirely of 
polynomially growing automorphisms. Below we provide examples where $\mathcal P = \mz^2$ or
$\mathcal P$ is a free group of rank $2$, see Section \ref{Examples}. We do not know how to 
deal with such groups if $\theta({\mathcal P})$ contains exponentially growing automorphisms.

\medskip

There is another way to extend Theorem \ref{short} above. Consider a group $G$ which is a cyclic 
extension of a hyperbolic group $\Gamma$. There exists an automorphism $\alpha$ of $\Gamma$ such
that the group $G$ is isomorphic to the semi-direct product $\Gamma \rtimes_{\alpha} \mz$. Again, 
the action of $G$ on $\Gamma$ by inner automorphisms extends to an action on the hyperbolic boundary 
$\partial \Gamma$ of $\Gamma$. We have:

\begin{theo}\label{hyperbolic}
Let $\Gamma$ be a torsion free, hyperbolic group with infinitely many ends.
Let $G = \Gamma \rtimes_{\alpha} \mz$ be a cyclic extension of $\Gamma$. 
Let $\mu$ be a non-degenerate probability measure on $G$ with a finite first moment. 
Then there exists a unique $\mu$-stationary probability 
measure $\lambda$ on $\partial \Gamma$ and the $G$-space $(\partial \Gamma,\lambda)$ is the Poisson 
boundary of $(G,\mu)$.
\par
The same conclusion holds if $\Gamma$ is the fundamental group $\pi_{1}(S)$ of a compact surface $S$
of genus $\geq 2$ and $\alpha$ is an exponentially growing automorphism.
\end{theo}

It is well known that automorphisms of the fundamental group of a compact surface $S$
of genus $\geq 2$ have either linear or exponential growth. Unfortunately, we are unable to
treat the case $G = \pi_{1}(S) \rtimes_{\alpha} \mz$ if $\alpha$ has linear (and non-zero) growth.

\medskip

Beyond the cases of abelian, nilpotent and free groups, there are many other 
classes of groups for which the Poisson boundary is already known. For instance, 
if $\mu$ is a finitely supported probability measure on a group $G$ with 
subexponential growth, then its Poisson boundary is trivial
\cite{Av74}. Discrete subgroups of $SL(d,\mr)$ are
treated in \cite{Fu71} where the Poisson boundary is related to
the space of flags. For random walks on Lie groups, the study
and the description of the Poisson boundary started in \cite{Fu63}
and was extensively developed in the 70's and the 80's by many
authors including Furstenberg \cite{Fu71,Fu73}, Azencott
\cite{Azencott}, Raugi and Guivarc'h \cite{Raugi77,Guiv-Raugi,Raugi88}, Ledrappier
\cite{Ledrappier,Ballm-Ledr}. Some analogies between
discrete groups and their continuous counterparts are enlightened,
see \cite{Fu71,Ledrappier,Ballm-Ledr}, but some
contrast can appear, compare  \cite{Raugi88} with \cite{KV,Kaim2}. The
Poisson boundary of some Fuchsian groups has also been described by
Series as being the limit set of the group \cite{Series}.
Kaimanovich and Masur \cite{KM} proved that the Poisson boundary
of the mapping class group is the boundary of Thurston's compactification
of the Teichm\"uller space. Their description also runs for the Poisson boundary 
of the braid group, see \cite{FaMas}.

\medskip

The Poisson boundary is closely related to some other well-known
compactifications or boundaries. The Martin boundary is
involved in the description of {\em positive} harmonic functions
(see \cite{Anc}). Considered as a measure space with the representing
measure of the constant harmonic function ${\bf 1}$, it is isomorphic to the
Poisson boundary \cite{Kaim96}. If the group $G$ has infinitely many ends (and if
$\mu$ is non-degenerate) then the Poisson boundary can be identified
with the space of ends (with the hitting measure), see
\cite{Woess,Cartwr-Soardi,Kaim1}. Furthermore, Kaimanovich
obtained an identification of the Poisson boundary for hyperbolic
groups with the geometric boundary, see \cite{Kaim1}. Concerning the
Floyd boundary \cite{Floyd}, Karlsson recently proved that if it is
not trivial, then it is isomorphic to the Poisson boundary, see
\cite{Karlsson}. The identification of these various
compactifications with the Poisson boundary strongly relies on a
general entropic criterion developed by Kaimanovich in \cite{Kaim1}
(and already used in \cite{Kaim2,KV,KM}). This
criterion will be our central tool in the present paper.

\medskip

Theorem \ref{short} reveals a certain rigidity property of the Poisson
boundary of the free group when one takes a cyclic extension. This phenomenon
does not occur for cyclic extension of free abelian groups. More precisely,
let $G\:=\:\mz^{d} \rtimes_{\alpha} \mz$ be the cyclic extension of $\mz^{d}$ 
by means of an automorphism $\alpha \in GL(d,\mz)$. If $\mu$ is any probability
measure on $G$ which projects to a recurrent random walk on $\mz$, then there exists 
a probability measure $\mu^0$ on $\mz^{d}$ such that the Poisson boundaries of $(G,\mu)$ 
and $(\mz^{d},\mu^0)$ are isomorphic \cite {Fu71} , hence trivial. 
On the other hand, if the induced random walk on $\mz$ is transient, then Kaimanovich 
proved that the Poisson boundary of $(G,\mu)$ is never trivial, see \cite{Kaim2} for a
precise description in terms of the contracting space of $\alpha$ acting on $\mc^d$.

\medskip

We now briefly outline the structure of the paper. In the next section, we describe 
the method we use to prove the theorems stated above. Some more detailed versions of 
these theorems are given. Sections \ref{Frederic} and \ref{generalites} contain all the material 
about harmonic functions, random walks, the Poisson boundary, as well as preliminaries about groups, 
automorphisms and semi-direct products that will be needed in the paper. The $\mr$-trees and affine actions
on them are introduced in Section \ref{rtrees}, and groups admiting such actions in Section 
\ref{class}. Sections \ref{exp} and \ref{cyclicpoly} contain the proofs of Theorems \ref{short} and
\ref{hyperbolic}, whereas Theorems \ref{polynomial} and \ref{direct} are proved in Section 
\ref{noncyclic}. Section \ref{Examples} contains two examples of non-cyclic groups of automorphisms 
of the free group, all of them having polynomial growth. Poisson boundary of relatively hyperbolic groups 
are discussed Section \ref{rh}.

\section{The method}\label{methode}

Let $\mu$ be a probability on a countable group $G$, and $X$ be a {\it compact} $G$-space endowed with a 
$\mu$-stationary probability $\lambda$. In \cite{Kaim1}, Vadim Kaimanovich provided 
sufficient conditions that should satisfy $(X,\lambda)$ to be the Poisson boundary of the pair $(G,\mu)$
(see Theorem \ref{theoreme 2.4-6.5} below). His results also contain informations on the asymptotic behaviour
of the random walk on $G$ governed by $\mu$. The aim of this section is to give a first flavour of the way
we construct such a compact $G$-space $X$ leading to two detailed and somehow technical statements 
(Theorems \ref{extrait1},\ref{extrait2} and \ref{extrait3} below) of which Theorems \ref{short}, \ref{polynomial} and 
\ref{direct} will appear as corollaries.

\medskip

Our approach is inspired by the long line of
works about free group automorphims starting at the beginning of the
nineties, see \cite{BestvinaHandel, BFH, GJLL, Gaboriau-Levitt, LL1,
LL2, CHL, CHLII} to cite only a few of them. Write $\F{d}$ for the rank $d$
free group, $d \geq 2$. A common feature to the
papers \cite{GJLL, Gaboriau-Levitt, LL1, LL2, CHL, CHLII} is the
introduction of a so-called ``$\alpha$-projectively invariant
$\F{d}$-tree'', that is a $\mr$-tree (a geodesic metric space such
that any two points are connected by a unique arc, and this arc is a
geodesic for the considered metric) with an isometric action of
$\F{d}$ and a homothety $H_\alpha$ satisfying the fundamental
relation: $H_\alpha(w.P) = \alpha(w).H_\alpha(P)$ (if the homothety
$H_\alpha$ is an isometry then the $\mr$-tree is
``$\alpha$-invariant''). Thanks to this relation, we have in fact an
action of the semi-direct product $\F{d} \rtimes_\alpha \mz$ on the
$\mr$-tree. Having in mind the problem of compactifying such a
semi-direct product, a major drawback of these $\mr$-trees is that
the usual adjonction of their Gromov-boundary does not provide us
with a compact space. This point is however settled by equipping
them with the so-called ``observers topology'': the union of the
completion of the $\mr$-tree $\mathcal T$ with its Gromov boundary
$\partial \mathcal T$, equipped with the observers topology, is
denoted by $\widehat{\mathcal T}$. This weak topology is thoroughly
studied in \cite{CHL} (see also \cite{Favre} where it appears in a
very different context).

\medskip

As mentioned in the introduction, we have two different kind of results, depending on the
nature of the action of the subgroup of automorphisms involved by
the extension: polynomially growing or exponentially growing
automorphisms. In Theorems \ref{extrait1} and \ref{extrait2} stated below, 
we extracted two particular cases from the various results of the paper (Theorems
\ref{resultat 1}, \ref{resultat 3}, \ref{resultat interessant 1bis},
\ref{resultat 2bis}, \ref{dernier resultat interessant}) to stress
this dichotomy.

\begin{theorem}
\label{extrait1}

Let $\F{d}$ be the rank $d$ free group, $d \geq 2$.
Let $\mathcal P$ be a finitely generated subgroup of
polynomially growing outer automorphisms of $\F{d}$. Let $\mu$ be a
probability measure on $\F{d} \rtimes \mathcal P$ the support of which
generates $\F{d} \rtimes \mathcal P$ as a semi-group. Then there exists a 
finite index subgroup $\mathcal U$ in $\mathcal P$ and a simplicial
$\mathcal U$-invariant $\F{d}$-tree $\mathcal T$ such that, if  ${\tau_{k}}$ denotes 
the (random) sequence of the times at which the path ${\bf x}=\{x_n\}$ visits 
$\F{d} \rtimes \mathcal U$, then:

\begin{itemize}
  \item Almost every sub-path $\{x_{\tau_{k}}\}$ of the random walk on 
$(\F{d} \rtimes \mathcal P,\mu)$ converges to some (random) limit $x_{\infty} \in \partial \mathcal T$.
  \item The distribution of $x_{\infty}$ is a non-atomic $\mu$-stationary measure $\lambda$ on
$\partial \mathcal T$; it is the unique $\mu$-stationary measure on $\partial \mathcal T$.
  \item If $\mu$ has finite first moment, then the measured space $(\partial {\mathcal T},\lambda)$ is the
Poisson boundary of $(\F{d} \rtimes \mathcal P,\mu)$.
\end{itemize}

\end{theorem}

\medskip

\begin{theorem}\label{extrait2}
Let $\F{d}$ be the rank $d$ free group, $d \geq 2$.
Let $G = \F{d} \rtimes_{\alpha} \Z$ be a cyclic extension of the free group $\F{d}$ over an
exponentially growing automorphism $\alpha$ of $\F{d}$. Let $\mu$ be a probability measure on $G$
the support of which generates $G$ as a semi-group. Then there exists an $\alpha$-projectively 
invariant $\F{d}$-tree $\mathcal T$ such that:

\begin{itemize}
  \item Almost every path $\{x_n\}$ of the random walk on $(G,\mu)$ 
converges to some limit $x_{\infty} \in \widehat{\mathcal T}$ (where 
$\widehat{\mathcal T}$ is the union of the completion of $\mathcal T$ with its Gromov boundary $\partial \mathcal T$, equipped with the observers topology).
  \item The distribution of $x_{\infty}$ is a non-atomic $\mu$-stationary measure $\lambda$ on
$\widehat{\mathcal T}$; it is the unique $\mu$-stationary measure on $\widehat{\mathcal T}$.
  \item If $\mu$ has finite first logarithmic moment and finite entropy, then the measured space
$(\widehat{\mathcal T},\lambda)$ is the Poisson boundary of $(G,\mu)$. The measure $\lambda$ is not concentrated on $\partial \mathcal T$. 
\end{itemize}

\end{theorem}

We describe the Poisson boundaries with more precision in the full
statements given farther in the paper. Note that, in the case of an
extension by polynomially growing automorphisms which are not trivial in $\Out{\F{d}}$, 
we need to pass to a finite index subgroup. That is the reason why we require the measure $\mu$ to
have a finite first moment, and why we obtain the convergence only for a subsequence
of a.e. path. See Theorem \ref{indice fini} below for more details. 
On the other hand, no such move to a finite index subgroup is needed in the
case of extensions by {\em inner} automorphisms, {\em i.e.} the case of a direct product.
We have:

\begin{theorem}\label{extrait3}
Let $G = \F{d} \times \F{k}$ be a direct product of two finitely generated free groups and $\beta$ a 
monomorphism $\beta \colon \F{k} \hookrightarrow \Inn{\F{d}}$. 
Let $\mu$ be a probability measure on $G$ the support of which generates $G$ as a semi-group. 
Then there exists a $\beta$-invariant $\F{d}$-tree $\mathcal T$ such that:

\begin{itemize}
  \item Almost every path $\{x_n\}$ of the random walk on $(G,\mu)$ 
converges to some limit $x_{\infty} \in \partial \mathcal T$.
  \item The distribution of $x_{\infty}$ is a non-atomic $\mu$-stationary measure $\lambda$ on
$\partial \mathcal T$; it is the unique $\mu$-stationary measure on $\partial \mathcal T$.
  \item If $\mu$ has finite first logarithmic moment and finite entropy, then the measured space
$(\partial \mathcal T,\lambda)$ is the Poisson boundary of $(G,\mu)$. 
\end{itemize}

\end{theorem}

We also deal with more
general hyperbolic groups than the free group. In the case of an
exponentially growing automorphism, the reader will notice that the
Poisson boundary is (a quotient of) the whole $\mr$-tree, and not
only the boundary of the tree. At least in the hyperbolic case, this will not seem
too surprising for geometric group theorists aware both of the
$\mr$-tree theory developed for surface and free group
automorphisms, and of the existence of the Cannon-Thurston map.
Indeed the whole $\mr$-tree is homeomorphic to a quotient of
$\partial \F{d}$ obtained by identifying the points which are the
endpoints of a same leaf of a certain ``stable lamination''
\cite{CHL}. The quotient we take amounts to further identifying the
points of $\partial \F{d}$ which are the endpoints of a same leaf of
a certain ``unstable lamination''. That this gives the geometric
boundary is known in the case of the suspension of a closed
hyperbolic surface by a pseudo-Anosov homeomorphism. In the more
general setting we work here we know no reference. It could perhaps
be alternatively obtained by combining the Cannon-Thurston map
defined in \cite{etudiantdeMitra} for the relatively hyperbolic
setting with the last section of the present paper and a work (yet
to be written) of Coulbois-Hilion-Lustig.

\medskip

As  was already written, we work really on the $\mr$-trees on
which the considered groups act, not on the groups themselves. This
leads us to extract the properties that we really need for these
actions to provide us with a compactification satisfying all the
Kaimanovich properties. In particular the map $\mathcal Q$
introduced in \cite{LL1, LL2} appears to play a crucial r\^ole. In
fact this map $\mathcal Q$ also allows us to get the following

\begin{corollary}\label{jackpot1}

Let $\F{d}$ be the rank $d$ free group, $d \geq 2$.
Let $G$ be an extension of the free group $\F{d}$ by a
finitely generated subgroup of polynomially growing outer
automorphisms. Let $\mu$ be a probability measure on $G$ the support of which
generates $G$ as a semi-group. Then there exists a finite index subgroup $H$
in $G$ such that, if  ${\tau_{k}}$ denotes the (random)
sequence of the times at which the path ${\bf x}=\{x_n\}$ visits $H$, then:

\begin{itemize}
   \item There exists a topology on $G \cup \partial \F{d}$ such
 that almost every sub-path ${x_{\tau_{k}}}$ converges to some $x_\infty \in \partial \F{d}$.
   \item The distribution of $x_{\infty}$ is a non-atomic $\mu$-stationary measure $\lambda$ on $\partial \F{d}$;
 it is the unique $\mu$-stationary measure on $\partial \F{d}$.
   \item If $\mu$ has finite first moment, then the measured space $(\partial \F{d},\lambda)$ is the
 Poisson boundary of $(G,\mu)$.
\end{itemize}

\end{corollary}

\medskip

\begin{corollary}\label{jackpot2}
Let $\F{d}$ be the rank $d$ free group, $d \geq 2$.
Let $G$ be either a cyclic extension $\F{d} \rtimes_{\alpha} \Z$ of the free group $\F{d}$ over an
exponentially growing automorphism $\alpha$ of $\F{d}$, or a direct product
$\F{d} \times \F{k}$ acting on $\partial \F{d}$ via some fixed monomorphism $\F{k} \hookrightarrow \Inn{\F{d}}$.
Let $\mu$ be a probability measure on $G$ the support of which generates $G$ as a semi-group. Then:

\begin{itemize}
   \item There exists a topology on $G \cup \partial \F{d}$ such
 that almost every path $\{x_n\}$ of the random walk converges to some $x_\infty \in \partial \F{d}$.
   \item The distribution of $x_{\infty}$ is a non-atomic $\mu$-stationary measure $\lambda$ on $\partial \F{d}$;
 it is the unique $\mu$-stationary measure on $\partial \F{d}$.
   \item If $\mu$ has finite first logarithmic moment and finite entropy, then the measured space
 $(\partial \F{d},\lambda)$ is the Poisson boundary of $(G,\mu)$.
\end{itemize}

\end{corollary}

Beware that $G \cup \partial \F{d}$ is {\it not}
a compactification of $G$. We invite the reader to compare the above
result with \cite[Theorems 1 and 2]{Vershik}. A perhaps more intuitive
interpretation of the above result in the case of a random walk
on a semi-direct product $\F{d} \rtimes_\alpha \mz$ is as follows: let $\F{d}$ be
the rank $d$ free group together with a basis $\mathcal B$
and let $\alpha$ be an automorphism of $\F{d}$. Consider a set of transformations $\mathcal S$
which consist either of a right-translation by an element in $\mathcal B$ or of the substitution
of an element $g$ by its image $\alpha(g)$ or $\alpha^{-1}(g)$. Then iterating randomly
chosen transformations among the set $\mathcal S$ amounts to performing a nearest neighbor
random walk on $\F{d} \rtimes_\alpha \mz$. The above corollary means that the boundary behavior
of this random process is entirely described in terms of the Gromov boundary of $\F{d}$.

\medskip

There is a difficulty, when dealing with direct products or
extensions over polynomially growing automorphisms, which does not
appear in the exponentially growing case and is hidden in the
statements given here. Indeed, whereas in the exponentially growing
case we only need to make act the considered group on a $\mr$-tree,
in the polynomially growing case we have to make act the group on
the product of two (simplicial) $\mr$-trees. The reason is that, in
order to check the Kaimanovich properties (more precisely the (CP)
condition), we need that a single element do not fix more than one
or two points, and if it fixes two then it acts as a hyperbolic
translation along the axis joining the two points. This is not so
difficult in the exponentially growing case because each element fixes at most
one point in the (completion of the) tree (exactly one for those acting as strict homotheties), the other ones lying in
the Gromov boundary. In the polynomially growing case, some elements
fix non-trivial subtrees of the considered $\mr$-tree. Collapsing
these subtrees is not possible because it might happen that
eventually everything gets identified.

\medskip

The trick is then
to make appear, as an intermediate step, a product of trees instead
of a single one. Since a semi-direct product structure only depends
on the outer-class of the automorphisms, we make the extension
considered act in different ways on two copies of the given
$\mr$-tree. In this way, the fixed points of the action which are an
obstruction to some of the Kaimanovich conditions become ``small'',
in some sense, in the ambient space. We mean that there is still a
tree to be identified to a single point for some elements but these
trees (in $\partial (\widehat{\mathcal T} \times \widehat{\mathcal
T})$) are disjoint for two distinct elements. We come back to a
single tree by projection on the first factor.

\medskip

Before concluding this section we would like to comment on possible 
generalizations of the results exposed here. When considering a group 
$G \rtimes {\mathcal U}$, the $\mr$-trees we work with are only a reminiscence 
of certain invariant laminations for the action of $\mathcal U$ on $G$. Here
the word ``invariant lamination'' has to be understood in the sense
of a geometric lamination or in the more general sense of an
algebraic lamination as in \cite{CHLI, CHLII, CHLIII}. Even in the
restricted geometric setting, such invariant laminations might exist
even if the manifolds considered are not suspensions. Thus one can
expect to be able to find more general classes of groups than those
considered here, like for instance in a first step the fundamental
groups of compact $3$-manifolds which do not fiber over $\SI$ but
nevertheless admit pseudo-Anosov flows (in the decomposition $G
\rtimes {\mathcal U}$ we do not need a priori that $G$ and $\mathcal
U$ be finitely generated, but only that $G \rtimes {\mathcal U}$
is).

\section{Harmonic functions and random walks on a discrete group}
\label{Frederic}

In this section, we recall the construction given by V.Kaimanovich of
the Poisson boundary of a countable group endowed with a probability
measure, together with a geometric characterization of this
boundary. Apart from subsections \ref{DS} and \ref{RM}, most of the material 
of this section is taken from \cite{Kaim1}. We end with a discussion on the 
first return measure on a recurrent subgroup.

\subsection{Stationary measures and random walks}

We write $\mn$ for the set of nonnegative integers. A (measurable)
$G$-space is any measurable space $(X,\Ff)$ measurably acted upon by
a countable group $G$. If $\mu$ and $\lambda$ are measures
respectively on $G$ and $X$, we denote by $\mu * \lambda$ the
measure on $X$ which is the the image of the product measure $\mu
\otimes \lambda$ by the action $G \times X \rightarrow X$. The
measure $\lambda$ is said to be {\em $\mu$-stationary} if one has:
\begin{equation}
\lambda\:=\:\mu * \lambda\:=\:\sum_{g}\mu(g)g \lambda\:\: .
\end{equation}

\medskip

Let $G$ be a countable group, and $\mu$ a probability measure on
$G$. The (right) {\em random walk} on $G$ determined by the measure
$\mu$ is the Markov chain on $G$ with the transition probabilities
$p(x,y) = \mu(x^{-1}y)$ invariant with respect to the left action of
the group $G$ on itself. Thus, the position $x_n$ of the random walk
at time $n$ is obtained from its position $x_0$ at time $0$ by
multiplying by independant $\mu$-distributed right increments $h_i$:
\begin{equation}\label{increment}
x_n\:=\:x_0 h_1 h_2 \cdots h_n.
\end{equation}

\medskip

Denote by $G^{\mn}$ the space of sample paths ${\bf x}=\{x_n\},
n\in\mn$ endowed with the $\sigma$-algebra $\Aa$ generated by the
cylinders $\{{\bf x}\in G^{\mn}\: |\: x_i = g\}$. The group $G$ acts
coordinate-wisely on the space $G^{\mn}$. An initial distribution
$\theta$ on $G$ determines the Markov measure $\PP_\theta$ on the
path space $G^{\mn}$ which is the image of the measure $\theta
\otimes \bigotimes^{\infty}_{n=1}\mu$ under the map
(\ref{increment}). The one-dimensional distribution of $\PP_\theta$
at time $n$, i.e. the distribution of $x_n$, is $\theta * \mu^{*n}$.

\medskip

In the sequel, we will be mainly interested in random walks starting
at the group identity $e$ which correspond to the initial
distribution $\theta = \delta_e$. We denote by $\PP$ the associated
Markov measure. For any initial distribution $\theta$, one easily
checks that the probability measure $\PP_{\theta}$ is equal to
$\theta * \PP$ and is absolutely continuous with respect to the
$\sigma$-finite measure $\PP_m$,
 where $m$ is the counting measure on $G$.

\subsection{The Poisson boundary}

The measure $\PP_m$ on the path space $G^{\mn}$ is invariant by the
time shift $T\: : \: \{x_n\} \mapsto \{x_{n+1}\}$. The {\em Poisson
boundary} $\Gamma$ of the random walk $(G,\mu)$ is defined as being
the space of ergodic components of the shift $T$ acting on the
Lebesgue space $(G^{\mn},\Aa,\PP_m)$ (see \cite{Rohlin}).

\medskip

Let us give some details. Denote by $\sim$ the orbit equivalence
relation of the shift $T$ on the path space $G^{\mn}$ :
\begin{equation}
{\bf x} \sim {\bf x'} \: \iff \: \exists \: n,n' \geq 0,\: T^{n}{\bf
x}=T^{n'}{\bf x'}\: .
\end{equation}
Let $\Aa_T$ be the $\sigma$-algebra of all the measurable unions of
$\sim$-classes, i.e. the $\sigma$-algebra of all $T$-invariant
measurable sets. Denote by $\overline{\Aa}_T$ the completion of
$\Aa_T$ with respect to the measure $\PP_m$. Since
$(G^{\mn},\Aa,\PP_m)$ is a Lebesgue space, the Rohlin correspondence
assigns to the complete $\sigma$-algebra $\overline{\Aa}_T$ a
measurable partition $\eta$ of $G^{\mn}$ called the {\em Poisson
partition}, which is well defined mod $0$. An atom of $\eta$
is an ergodic component of the shift $T$, that is, up to a set of
$\PP_m$-measure $0$, closed under $\sim$, $\Aa$-measurable, and minimal
with respect to these propoerties. (Note that the $\sigma$-algebra
$\Aa_T$ is not {\em a priori} generated by the atoms of $\eta$, see
\cite{Rohlin,Babillot}). The Poisson boundary $\Gamma$ is
the quotient space $G^{\mn}/\eta$. The coordinate-wise action of $G$
on the path space $G^{\mn}$ commutes with the shift $T$ and therefore projects
to an action on $\Gamma$. Denote by bnd the canonical map bnd $ : \:
G^{\mn}\: \rightarrow \: \Gamma$. The space $\Gamma$ endowed with
the bnd-image of the $\sigma$-algebra $\overline{\Aa}_T$ and the
measure $\nu_{m}\:=\:$bnd$(\PP_m)$ is a Lebesgue space.

\medskip

For any initial distribution $\theta$, we set $\nu_{\theta}\: =
\:$bnd$(\PP_{\theta})$. The Poisson boundary $\Gamma$, which depends
only on $G$ and $\mu$, carries all the probability measures
$\nu_{\theta}$. The measure $\nu = $bnd$(\PP)$ is called the {\em
harmonic measure}. One easily checks that $\nu_{\theta}\: =
\:$bnd$(\theta * \PP)\: = \: \theta * \nu$. It implies that the
measure $\nu$ is $\mu$-stationary, i.e. $\nu  \:=\: \mu * \nu$
(whereas the other measures $\nu_{\theta}$ are not).

\subsection{Harmonic functions}

As mentioned in the introduction, the space $(\Gamma,\nu)$ enables
to retrieve both all the bounded $\mu$-harmonic functions on $G$ and
(part of) the asymptotic behaviour of the paths of the random walk.
Let us now make this precise.

\medskip

The Markov operator $P\:=\:P_{\mu}$ of averaging with respect of the
transition probability of the random walk $(G,\mu)$ is defined by
\begin{equation}
P_{\mu}f(x)\:=\:\sum_{y}p(x,y)f(y)\:=\:\sum_{h}\mu(h)f(xh)\:.
\end{equation}
A function $f:G\rightarrow \mr$ is called $\mu$-harmonic if
$Pf\:=\:f$. Denote by $H^{\infty}(G,\mu)$ the Banach space of
bounded $\mu$-harmonic functions on $G$ equipped with the sup-norm.

\medskip

There is a simple way to build bounded $\mu$-harmonic functions on
$G$. Assume that one is given a probability $G$-space
$(X,\Ff,\lambda)$ such that the measure $\lambda$ is
$\mu$-stationary. To any function $F\in L^{\infty}(X,\lambda)$ one
can assign a function $f$ on $G$ defined, for $g\in G$, by the {\em
Poisson formula} $f(g)\:=\:\langle F,g\lambda
\rangle\:=\:\int_{X}F(gx)d\lambda(x)$. Since
$\lambda\:=\:\mu*\lambda$, the function $f$ is $\mu$-harmonic.

\medskip

In the case of the Poisson boundary $(\Gamma,\nu)$, the Poisson
formula is indeed an isometric isomorphism from
$L^{\infty}(\Gamma,\nu)$ to $H^{\infty}(G,\mu)$. Actually, one can
prove (see \cite[theorem 6.1]{Kaim1bis}) that if $f$ is the harmonic
function provided by the Poisson formula applied to a function $F
\in L^{\infty}(\Gamma,\nu)$, then for $\PP$-almost every path ${\bf
x}=\{x_n\}$, one has $F($bnd ${\bf x})\:=\:\lim f(x_n)$. Conversely,
if $f$ is any bounded $\mu$-harmonic functions on $G$, then the
sequence of its values along sample paths ${\bf x}=\{x_n\}$ of the
random walk is a martingale (with respect to the increasing
filtration of the coordinate $\sigma$-algebras in $G^{\mn}$).
Therefore, by the Martingale Convergence Theorem, for $\PP$-almost
every path ${\bf x}=\{x_n\}$, there exists a limit ${\hat F}({\bf
x})\:=\:\lim f(x_n)$. This function ${\hat F}$ is $T$-invariant and
$\Aa_T$-measurable. Since the Poisson boundary $\Gamma$ is the
quotient of the path space determined by $\Aa_T$, there exists $F\in
L^{\infty}(\Gamma,\nu)$ such that $F($bnd ${\bf x})\:=\:\lim
f(x_n)$.

\subsection{$\mu$-Boundaries}

As far as the behaviour of sample paths at infinity is concerned, we
first recall the notion of {\em $\mu$-boundary}. This notion was
first introduced by Furstenberg, see \cite{Fu71,Fu73}.
The following definition is due to Kaimanovich.
 A $\mu$-boundary is a $G$-equivariant quotient
$(B,\lambda)$ of the Poisson boundary $(\Gamma,\nu)$, i.e. the
quotient of $\Gamma$ with respect to some $G$-equivariant measurable
partition. The Poisson boundary is itself a $\mu$-boundary, and this
space is maximal with respect to this property. Particularly, if
$\pi$ is any $T$-invariant $G$-equivariant map from the path space
$(G^{\mn},\PP)$ to some $G$-space $B$, then, by definition of the
Poisson boundary, such a map factors through $\Gamma$:
$\pi:G^{\mn}\:\rightarrow\: \Gamma \:\rightarrow B$ and
$(B,\pi(\PP))$ is a $\mu$-boundary.

\medskip

Furstenberg's construction of $\mu$-boundaries runs as follows.
Assume that $B$ is a separable compact $G$-space on which $G$ acts
continuously. By compactness, there exists a $\mu$-stationary
measure $\lambda$ on $B$ (see \cite{KM}). Then the Martingale
Convergence Theorem implies that for $\PP$-almost every path ${\bf
x}=\{x_n\}$, the sequence of translations $x_{n}\lambda$ converges
weakly to some measure $\lambda({\bf x})$. Therefore the map ${\bf
x}\rightarrow \lambda({\bf x})$ allows to consider the space of
probability measures on $B$ as a $\mu$-boundary. If, in addition,
the limit measures $\lambda({\bf x})$ are Dirac measures, then the
space $B$ is itself a $\mu$-boundary (see \cite{Fu71,Fu73}).

\medskip

For instance, assume that the group $G$ contains a finitely generated free group $F$
as a normal subgroup. Then the action of $G$ on $F$ by inner
automorphisms extends to a continuous action on the boundary $\partial F$
of $F$. Vershik and Malyutin \cite{Vershik} proved that if $\mu$
is any nondegenerate measure on $G$ (i.e. the support of $\mu$
generates $G$ as a semigroup) then there exists a unique $\mu$-stationary
measure $\lambda$ on the $G$-space $\partial F$ and $(\partial F,\lambda)$
is a $\mu$-boundary of the pair $(G,\mu)$. Their proof heavily relies on the
description of some contracting properties of the action of $G$ on $\partial$. 
The aim of this paper is to provide various situations for which we can prove 
that $(\partial F,\lambda)$ is indeed the Poisson boundary of $(G,\mu)$.

\subsection{Dynkin spaces}\label{DS}

The concept of {\em Dynkin space} was introduced by Furstenberg in \cite{Fu67} in order to prove that a given $G$-space is
a $\mu$-boundary. The purpose of this subsection is to discuss this notion in the case of cyclic extensions of the free group.
A Dynkin space for a finitely generated group $G$ is a compact metric $G$-space $X$ with the following property :
\[
\forall \varepsilon > 0, \exists \:N\: \mbox{s.t.} \forall g\in G, |g| \geq N, \exists U_{g},V_{g} \subset X,\:
\mbox{diam}\:U_{g}\:\mbox{and}\:V_{g}\:< \varepsilon \:\:\mbox{s.t.}\: g(U_{g}^{c}) \subset V_{g} \:\: .
\]
One can easily check that the boundary $\partial F$ of a free group $F$ is a Dynkin space for $F$ (for the usual
ultra-metric on $\partial F$). If $G = F \rtimes_{\alpha} \Z$ is a cyclic extension of a free group $F$ by means of
some automorphism $\alpha$ of $F$, then, up to finite index, $\partial F$ is also a Dynkin space for $G$. For, if 
$\partial \alpha$ denotes the continuation of $\alpha$ on  $\partial F$, then $\partial \alpha$ has finitely many periodic
orbits in $\partial F$, and the orbit of any element accumulates on a periodic orbit \cite{LevLus08}. Therefore, up to taking
some power of $\alpha$ (which amounts to considering a finite index subgroup of $G$), one may consider that the periodic orbit 
of $\alpha$ are fixed points, and the orbit of any non-fixed point is attracted to a fixed point.

\medskip

However, for a $G$-space $X$, being a Dynkin space for $G$ is not sufficient to ensure that $X$ is a $\mu$-boundary.
Actually, in \cite{Fu67}, Furstenberg strongly relies on the fact that $\partial F$ is a {\em compactification} of the free group $F$
to prove both the fact that $\partial F$ is a $\mu$-boundary and also the convergence of the random walk in $F$ to $\partial F$.
In the case of a cyclic extension $G = F \rtimes_{\alpha} \Z$ of $F$, we do not know {\em a priori} how to compactify $G$ with 
$\partial F$ in such a way to prove that $\partial F$ is a $\mu$-boundary of $G$, and eventually its Poisson boundary.
Indeed, the main achievement of this paper is to construct a compactification $\overline{G}$ of $G$ containing $\partial F$ as a dense 
subset of full measure in such a way that $G \cup \partial F$ has all the properties listed in Corollaries \ref{jackpot1} and \ref{jackpot2}.

\subsection{Compactifications}

The situation we are mainly concerned with is the following. Let
$\overline{G}$ be a {\em compactification} of the group $G$, that is
a topological compact space which contains $G$ as an open dense
subset and such that the left-action of $G$ on itself extends to a
continuous action on $\overline{G}$. Assume that $\PP$-almost every
path ${\bf x}=\{x_n\}$ converges to a limit $x_{\infty}\:=\:\lim
x_{n}\:=\:\pi({\bf x})\:\in \overline{G} $. Then the map $\pi$ is
obviously $T$-invariant and $G$-equivariant, so that the space
$\overline{G}$ equipped with the hitting measure $\pi(\PP)$ is a
$\mu$-boundary. Moreover, in this case, the Poisson formula yields
an isometric embedding of  $L^{\infty}(\overline{G},\pi(\PP))$ into
$H^{\infty}(G,\mu)$.

\medskip

A compactification $\overline{G}$ is {\em $\mu$-maximal} if $\PP$-almost
every sample path ${\bf x}=\{x_n\}$ of the random walk $(G,\mu)$
converges in this compactification to a (random) limit $x_{\infty} \: =\: \pi({\bf x})$, 
and if $(\overline{G},\lambda)$ is isomorphic to the
Poisson boundary of $(G,\mu)$, where $\lambda\: = \: \pi(\PP)$ is the hitting measure .

\medskip

In \cite{Kaim1}, Kaimanovich proves a theorem which provides
compactifications $\overline{G}$ of $G$ containing the limit of
$\PP$-almost every path (theorem 2.4), and another one which is a
geometric criterion of maximality of $\mu$-boundaries (theorem 6.5).
The combination of these two results yields $\mu$-maximal
compactifications of $G$, see Theorem \ref{theoreme 2.4-6.5} below.
This statement will be our central tool.

\medskip

A compactification $\overline{G}$ is {\em compatible} if the
left-action of $G$ on itself extends to an action on $\overline{G}$
by homeomorphisms. A compactification is {\em separable} if, when
writing $\overline{G} = G \cup \partial G$, then $\partial G$ is
separable (i.e contains a countable dense subset). In the remaining
of this section, we shall always assume that $\overline{G}$ is a
compatible and separable compactification of the group $G$.

\medskip

We first state a uniqueness criterion. Its proof is mostly inspired
from the one of Theorem 2.4 in \cite{Kaim1}.
\begin{lemma}\label{unique}
Let $\overline{G} = G \cup \partial G$ be a compactification of a finitely
generated group $G$ satisfying the following {\em proximality property:} whenever
$G \ni g_{n} \rightarrow \xi \in \partial G$, one has
$g_{n}\eta \rightarrow \xi$ for all but at most one $\eta \in \partial G $.
Let $\mu$ be a probability measure on $G$ such that the subgroup
$gr(\mu)$ generated by its support fixes no finite subset of
$\partial G$. Assume that $\PP$-almost every path ${\bf x}=\{x_n\}$ converges
to a (random) limit $x_{\infty} \: =\: \pi({\bf x})\:\in \partial G$.
Then the hitting measure $\lambda \: = \: \pi(\PP)$, which is the
distribution of $x_{\infty}$, is non-atomic and is the unique $\mu$-stationary
probability measure on $\partial G$.
\end{lemma}

\begin{proof}
Let $\nu$ be any $\mu$-stationary measure on $\partial G$. The hypothesis on
$gr(\mu)$ ensures that $\nu$ is non-atomic (see the proof of Theorem 2.4 in \cite{Kaim1}
for details). Then the proximality property implies that the sequence
of (random) measures $x_{n}\nu$ weakly converges to the Dirac measure $\delta_{x_{\infty}}$.
Since $\nu$ is $\mu$-stationary and $\mu^{*n}$ is the distribution of $x_n$,
we have
\[
\nu \:=\:\mu^{*n}*\nu \:=\: \sum_{g}\mu^{*n}(g).g\nu \:=\: \int x_{n}\nu d\PP({\bf x}) \:.
\]
Passing to the limit on $n$ gives
\[
\nu \:=\: \int \delta_{x_{\infty}} d\PP({\bf x}) \:=\: \lambda\:.
\]
\end{proof}

\begin{remark}
\label{compacite non necessaire}
Note that there is no need of compactness in the proof: Lemma \ref{unique}
remains valid if $\overline{G}$ is not compact. This observation will be used later in the paper.
\end{remark}

\subsection{Kaimanovich's criterion of maximality}

The compactification $\overline{G}$ satisfies {\em condition (CP)}
if, for any $x\in G$ and for every sequence $g_n \in G$ converging
to a point from $\partial G$ in the compactification $\overline{G}$,
the sequence $g_{n}x$ converges to the same limit.

\medskip

The compactification $\overline{G}$ satisfies {\em condition (CS)}
if the following holds. The boundary $\partial G$ consists of at
least 3 points, and there is a $G$-equivariant Borel map $S$
assigning to pairs of distinct points $(b_{1},b_{2})$ from $\partial
G$ nonempty subsets ({\em strips}) $S(b_{1},b_{2}) \subset G$ such
that for any 3 pairwise distinct points $\overline{b_{i}} \in
\partial G$, $i\: =\: 0,1,2$, there exist neighbourhoods
$\overline{b_{0}} \in \Oo_{0} \subset \overline{G}$ and
$\overline{b_{i}} \in \Oo_{i}
 \subset \partial G$, $i\: =\: 1,2$ with the property that
\[
S(b_{1},b_{2}) \: \cap \: \Oo_{0}\: = \: \emptyset \quad\quad
\text{for all}\quad b_{i} \in \Oo_{i},\: i\: =\: 1,2 \: .
\]
This condition (CS) means that points from $\partial G$ are
separated by the strips $S(b_{1},b_{2})$.

\medskip

A {\em gauge} on a countable group $G$ is any increasing sequence
$\Jj \: = \: (\Jj_{k})_{k\geq 1}$ of sets exhausting $G$. The
corresponding {\em gauge function} is defined by $|g| \: = \:
|g|_{\Jj} \: = \: \text{min}\{k\: , \: g\in \Jj_{k} \}$. The gauge
$\Jj$ is {\em finite} if all gauge sets are finite. An important
class of gauges consists of {\em word gauges}, i.e. gauges
$(\Jj_{k})$ such that $\Jj_{1}$ is a set generating $G$ as a
semigroup, and $\Jj_{k} \: = \: (\Jj_{1})^{k}$ is the set of word of
length $\leq \: k$ in the alphabet $\Jj_{1}$. It is finite if and
only if $\Jj_{1}$ is finite. A set $S \subset G$ {\em grows
polynomially} with respect to some gauge $\Jj$ if
\[
\mbox{there exist}\: A,B,d \:\:\mbox{such that}\:\: card\: [S \cap
\Jj_{k}] \leq A\: + \: Bk^d \:.
\]

\medskip

\begin{theorem}\cite[Theorems 2.4 and 6.5]{Kaim1}
\label{theoreme 2.4-6.5} Let $G$ be a finitely generated group. Let
$\overline{G} = G \cup \partial G$ be a compatible and separable
compactification of $G$ satisfying conditions (CP) and (CS). Let
$\mu$ be a probability measure on $G$ such that the subgroup
$gr(\mu)$ generated by its support fixes no finite subset of
$\partial G$.

Then $\PP$-almost every path ${\bf x}=\{x_n\}$ converges to a
(random) limit $x_{\infty} \: =\: \pi({\bf x})\:\in \partial G$. The
hitting measure $\lambda \: = \: \pi(\PP)$ is non-atomic, the measure
space $(\partial G,\lambda)$ is a $\mu$-boundary and $\lambda$ is the
unique $\mu$-stationary probability measure on $\partial G$.

Moreover, if $\Jj$ is a finite word gauge on $G$ such that the
measure $\mu$ has a finite first logarithmic moment $\sum \log
|g|\mu(g)$, a finite entropy $H(\mu)\: = \: -\sum \mu(g) \log
\mu(g)$, and if each strip $S(b_{1},b_{2})$ grows polynomially, then
the space $(\partial G,\lambda)$ is isomorphic to the Poisson
boundary $(\Gamma,\nu)$ of $(G,\mu)$ and is therefore $\mu$-maximal.
\end{theorem}

\medskip

Note that if a measure $\mu$ has a finite first moment $\sum
|g|\mu(g)$, then the entropy $H(\mu)$ is finite, see e.g.
\cite[Lemma 12.2]{Kaim1bis}. Besides, the finitess of the first
(logarithmic) moment does not depend on the choice of the finite word
gauge.

\medskip

\subsection{The first return measure}\label{RM}

This subsection is devoted to a discussion on the stability of the Poisson
boundary when moving to a subroup which is recurrent or normal with
finite index. Let $G$ be a finitely generated group, $\mu$ a probability
measure on $G$ and $G^{0}$ a subgroup of $G$ which is a recurrent set
for the random walk $(G,\mu)$. Define a probability measure $\mu^{0}$
on $G^{0}$ as the distribution of the point where the random walk issued
from the identity of $G$ returns for the first time to $G^{0}$. We call
$\mu^{0}$ the {\it first return measure}.
Furstenberg observed (see \cite[Lemma 4.2]{Fu71}) that the Poisson
boundaries of $(G,\mu)$ and $(G^{0},\mu^{0})$ are isomorphic.

\medskip

For instance, if $G^{0}$ is a normal subgroup in $G$ such that the
random walk $(G,\mu)$ projects on the factor group $G/G^{0}$ as
a recurrent random walk, then the identity in $G/G^{0}$ is a recurrent
state, therefore $G^{0}$ is a recurrent set in $G$. The case of a normal
subgroup of finite index is of special interest:

\begin{lemma}\cite[Lemma 4.2]{Fu71},\cite[Lemma 2.3]{Kaim2}
\label{merci Kaima} Let $G$ be a finitely generated group, $\mu$ a
probability measure on $G$ and $G^{0}$ a normal subgroup of finite
index in $G$. Then $G^{0}$ is a recurrent set for the random walk
$(G,\mu)$ and the Poisson boundaries of $(G,\mu)$ and $(G^{0},\mu^{0})$
are isomorphic where $\mu^{0}$ is the first return measure.
Moreover, if $\mu$ has a finite first moment (with respect to some
finite word gauge) then so has $\mu^{0}$.
\end{lemma}

Observe that the conclusion of this lemma remains valid if the
finite index subgroup $G^{0}$ is not normal, since it contains a
subroup $G^{1}$ of finite index which is normal in $G$.

\medskip

Indeed, what is meant in this lemma is that there is an isomorphism between
the spaces of harmonic functions $H^{\infty}(G,\mu)$ and $H^{\infty}(G^{0},\mu^{0})$.
But the Poisson boundaries $(\Gamma,\nu)$ and $(\Gamma^{0},\nu^{0})$ of
$(G,\mu)$ and $(G^{0},\mu^{0})$ are distincts objects in nature since $\Gamma^{0}$
is not {\it a priori} a $G$-space - although $\Gamma$ and $\Gamma^{0}$ are both $G^{0}$-spaces. 
This is made precise in the following statement,
which is a combination of Theorem \ref{theoreme 2.4-6.5} and (the proof of) Lemma \ref{merci Kaima}:

\begin{theorem}\label{indice fini}
Let $G$ be a finitely generated group and $G^{0}$ a finite index subgroup of $G$.
Let $\overline{G^{0}} = G^{0} \cup B$ be a compatible and separable
compactification of $G^{0}$ satisfying conditions (CP) and (CS) such that
the action of $G^{0}$ on $B$ extends to a continuous action of $G$ on $B$.
Let $\mu$ be a probability measure on $G$ such that the subgroup
$gr(\mu)$ generated by its support fixes no finite subset of $B$.
Denote by $\mu^{0}$ the first return measure. Let ${\tau_{k}}$ be the (random)
sequence of the times at which the path ${\bf x}=\{x_n\}$ visits $G^{0}$.
Then :

\begin{enumerate}
\item $\PP$-almost every sub-path ${x_{\tau_{k}}}$ converges to a
(random) limit $x_{\infty} \: =\: \pi({\bf x})\:\in B$.
\item  the hitting measure $\lambda \: = \: \pi(\PP)$ is non-atomic, the measure
space $(B,\lambda)$ is a $\mu$-boundary and $\lambda$ is the
unique $\mu$-stationary probability measure on $B$.
\item if $\Jj$ is a finite word gauge on $G$ such that the measure $\mu$ has a
finite first moment and each strip grows polynomially, then
the space $(B,\lambda)$ is isomorphic to the Poisson boundary $(\Gamma,\nu)$ of $(G,\mu)$.
\end{enumerate}

\end{theorem}

\begin{proof}
Since $G^{0}$ has finite index in $G$, the subgroup $gr(\mu^{0})$ of $G^{0}$ generated by
the support of $\mu^{0}$ fixes no finite subset of $B$. Therefore Theorem \ref{theoreme 2.4-6.5}
works for the random walk $(G^{0},\mu^{0})$, which gives Item $(1)$, as well as Items $(2)$ and $(3)$
for $\mu^{0}$.

\medskip

Let us prove that $\lambda$ is the unique $\mu$-stationary probability measure on $B$. Let $\nu$ be
any $\mu$-stationary probability measure on $B$ (the compactness of $B$ implies the existence of such
a measure). Consider the Poisson formula $\Pi_{\nu} : L^{\infty}(B,\nu) \rightarrow H^{\infty}(G,\mu)$.
According to the proof of Lemma 4.2 in \cite{Fu71}, the restriction $\Phi : f \mapsto f_{|G^{0}}$ maps
$\mu$-harmonic functions to $\mu^{0}$-harmonic functions. Therefore the composition $\Phi \circ \Pi_{\nu}$
is the Poisson formula $L^{\infty}(B,\nu) \rightarrow H^{\infty}(G^{0},\mu^{0})$ so the measure $\nu$
is $\mu^{0}$-stationary, and $\nu = \lambda$.

\medskip

As far as Item $(3)$ is concerned, we prove that the Poisson formula
$\Pi_{\lambda} : L^{\infty}(B,\lambda) \rightarrow H^{\infty}(G,\mu)$
is an isomorphism. The map $\Phi : H^{\infty}(G,\mu) \rightarrow H^{\infty}(G^{0},\mu^{0})$ is an
isomorphism. Since $\mu$ has a first moment, so has $\mu^{0}$. According to Item (3) of Theorem
\ref{theoreme 2.4-6.5}, $(B,\lambda)$ is the Poisson boundary of $(G^{0},\mu^{0})$, therefore the composition
$\Phi \circ \Pi_{\lambda}$ is an isomorphism, and so is $\Pi_{\lambda}$.
\end{proof}

Observe that the way the group $G$ acts on $B$ does not play any role, provided this action
extends the action of $G^0$. Indeed, the boundary behaviour of the random walk on $(G,\mu)$
is governed by $(G^{0},\mu^{0})$.

\medskip

Assume now the following: $G^{0}$ is the free group $\F{d}$, the
factor group $G/\F{d}$ is isomorphic to $\Z$ and the
image random walk on $G/\F{d}$ is recurrent. Then the above
construction gives rise to a probability measure $\mu^{0}$ on
$\F{d}$ such that the Poisson boundaries of $(G,\mu)$ and
$(\F{d},\mu^{0})$ are isomorphic. The problem is that, in this
case, we do not know {\it a priori} what the Poisson boundary of
$(\F{d},\mu^{0})$ should be, because the measure $\mu^{0}$ is potentially too
spread out to have a finite entropy, even if the image random
walk is the symmetric nearest neighbour random walk on $\Z$.
On the other hand, according to Theorem \ref{short}, the Poisson
boundary of $(G,\mu)$ - and that of $(\F{d},\mu^{0})$ - is $\partial\F{d}$.

\medskip

Besides, if $G$ is, for instance, an extension of the free group $\F{4}$
by a $\Z^2$-subgroup of polynomially growing automorphisms in $\Aut{\F{4}}$
(see Subsection \ref{exemple2} for such examples) and if $\mu$ is any
non-degenerate probability measure with finite first moment on $G$ which leads to a recurrent random walk on
$G/\F{4} \sim \Z^2$ then the measure $\mu^{0}$ on $\F{4}$ is potentially
{\it very} spread out, and Corollary \ref{jackpot1} ensures that
the Poisson boundary of $(\F{4},\mu^{0})$ is $\partial\F{4}$.

\section{Generalities about groups, automorphisms and semi-direct
products}\label{generalites}

Let $G$ be a discrete group with generating set $S$, that we denote
by $G = \langle S \rangle$. The Cayley graph of $G$ with respect to
$S$ is denoted by $\Gamma_S(G)$. It is equipped with the standard
metric which makes each edge isometric to $(0,1)$. We denote by
$|\gamma|_S$ the word-length of an element $\gamma$ with respect to
$S$, i.e. the minimal number of elements in $S \cup S^{-1}$
necessary to write $\gamma$. If it is clear, or not important, which
generating set is used, we will simply write $|\gamma|$ for the
word-length of $\gamma \in G$ (in particular, when dealing with free
groups, $|\gamma|$ will denote the word-length with some fixed
basis).

We denote by $\Aut{G}$ the group of automorphisms of $G$, by
$\Inn{G}$ the group of inner automorphisms (i.e. automorphisms of
the form $i_g(x) = g x g^{-1}$) and by $\Out{G} = \Aut{G} / \Inn{G}$
the group of outer automorphisms. While $\Aut{G}$ acts on elements
of $G$, $\Out{G}$ acts on conjugacy-classes of elements. If
$\mathcal U < \Aut{G}$ is a subgroup, we denote by $[\mathcal U] <
\Out{G}$ its image under the canonical projection.

\begin{definition} \label{growth}
An automorphism $\alpha$ of a $G$ has {\em polynomial growth} if
there is a polynomial function $P$ such that, for any $\gamma$ in
$G$, for any $m \in \mn$, $|\alpha^m(\gamma)| \leq P(m) |\gamma|$.

We say that $\alpha$ has {\em exponential growth} if the lengths of
the iterates $\alpha^m(\gamma)$ of at least one element $\gamma$
grow at least exponentially with $m \to + \infty$.
\end{definition}

The above notions only depend on the outer-class of automorphisms
considered. Passing from $\alpha$ to $\alpha^{-1}$ neither changes the nature
of the growth.

Let $\theta \colon {\mathcal U} \rightarrow \Aut{G}$ be a
monomorphism. We denote by $G_\theta := G \rtimes_\theta {\mathcal
U}$ the semi-direct product of $G$ with $\mathcal U$ over $\theta$.
The semi-direct product only depends on the outer-class of
$\theta({\mathcal U})$: $G_\theta$ is isomorphic to
$G_{\theta^\prime}$ whenever $[\theta({\mathcal U})] =
[\theta^\prime({\mathcal U})]$ in $\Out{G}$. For this reason we will
also write a semi-direct product $G_\theta := G \rtimes_\theta
\mathcal U$ for $\theta \colon \mathcal U \rightarrow \Out{G}$.

We will denote by $\F{n} = \langle x_1,\cdots,x_n \rangle$ the rank
$n$ free group. A {\em hyperbolic group} is a finitely generated
group $G = \langle S \rangle$ such that there exists $\delta \geq 0$
for which the geodesic triangles of $\Gamma_S(G)$ are $\delta$-thin:

{\em For any triple of geodesics $[x,y], [y,z], [x,z]$ in
$\Gamma_S(G)$, $[x,z] \subset {\mathcal N}_\delta([x,y] \cup
[y,z])$, where ${\mathcal N}_\delta(X)$ denotes the set of all
points at distance smaller than $\delta$ from some point in $X$.}

We briefly recall the definition of the Gromov boundary $\partial G$
of a hyperbolic group $G$ \cite{Gromov}, which allows one to get a geodesic
compactification $G \cup \partial G$ of $G$. For more details, see
\cite{benaklikapovich}. The Gromov boundary $\partial G$ of the
hyperbolic group $G$ consists of the set of equivalence classes of geodesic rays emanating
from the base-point, two rays being equivalent if their Hausdorff
distance is finite. A {\em ray emanating from $O$} is the interior of (the image of) an
isometric embedding $\rho \colon [0,+\infty) \rightarrow
\Gamma_S(G)$ with $\rho(0) = O$. Of course the definition of a
Gromov boundary naturally extends to any Gromov hyperbolic metric
space (a geodesic metric space whose geodesic triangles are
$\delta$-thin). In the particular case of a tree $T$, which will be
frequent in this paper, two geodesic rays in $T$ define a same point
in $\partial T$ if they eventually coincide. Beware that the Gromov
boundary $\partial X$ of a geodesic metric space $X$ {\em does not}
provide us with a compactification of $X$ if $X$ is not proper (where
``proper'' here means that the closed balls are compact). For
instance, if $T$ is a locally infinite tree, then $T \cup \partial
T$ is not a compactification of $T$.

\section{About $\mr$-trees and group actions on $\mr$-trees}\label{rtrees}

The aim of this section is to gather all the vocabulary as well as
all the notions and results that we will need about $\mr$-trees and
group actions on $\mr$-trees further in the paper. The reader may
choose to skip this chapter, only coming back each time subsequent
sections refers to the material developed below. We counsel however
to have a quick look at the first part, which motivates the
introduction of these $\mr$-trees.

\subsection{How $\mr$-trees come into play}

\label{une introduction introductive}

We take a pedestrian way to show to the reader how $\mr$-trees
naturally come into play when searching for Poisson boundaries of
groups extensions. We consider here the group $\F{n} \times \mz$.
Let $\F{n} = \langle x_1,\cdots,x_n \rangle$ and $\mz = \langle t
\rangle$. We assume it is equipped with a probability measure $\mu$
and we would like to find a compactification of $\F{n} \times \mz$
suitable for applying Kaimanovich tools.

The most natural space on which $\F{n} \times \mz$ acts is of course
$T \times \mz$, where $T$ is the Cayley-graph  of $\F{n}$ with
respect to some basis. This is a simplicial, locally finite tree.
For compactifying $T \times \mz$, if $v$ is a vertex of $T$ we need
that $t^{k_i} v$ converges to some point $O_v$ for any vertex of
$T$. For the (CP) property to be satisfied, we need that the
limit-point $O_v$ be the same for all $v$. This point would then be
fixed by all elements of $\F{n} \times \mz$, so that Theorem
\ref{theoreme 2.4-6.5} would not apply. The cause of the problem
here holds in the fact that the $\mz$-action on $\F{n}$ has many
fixed points \ldots

We are now going to take advantage from the fact that a direct
product is a particular case of a semi-direct product: $\F{n} \times
\mz$ is isomorphic to $\F{n} \rtimes_\alpha \mz$ with $\alpha \in
\Inn{G}$.  We define $\alpha \in \Inn{\F{n}}$ by $\alpha(x_i) = x_1
x_i x^{-1}_1$ for $i \in \{1,\cdots,n\}$. The Cayley-graph of $\F{n}
\rtimes_\alpha \mz$ is now the $1$-skeleton of $\bigsqcup_{t \in
\mz} T \times [t,t+1] / ((x,t+1) \in T \times [t,t+1] \sim
(f(x),t+1) \in T \times [t+1,t+2])$, where $f \colon T \rightarrow
T$ is a PL-map realizing $\alpha$, i.e. the image under $f$ of an
edge of $T$ with label $x_i$ is a PL-path in $T$ with associated
edge-path $\alpha(x_i)$. We observe that, contrary to what happened
before, there are now two kinds of asymptotic behavior for the
$\alpha^k(w)$'s: for any $w \neq x_1$, $\displaystyle \lim_{k \to
\pm \infty} \frac{|\alpha^k(w)|}{|k|} = 2$ whereas $\displaystyle
\lim_{k \to \pm \infty} \frac{\alpha^k(x_1)}{|k|} = 0$. In other
words, for $w \neq x_1$, the length of the $\alpha^k(w)$ tend
linearly toward infinity with $k$ whereas the lengths of the
$\alpha^k(x_1) = x_1$ remain constant equal to $1$.

Having still in the mind the compactification of the Cayley graph,
since some of the orbits separate linearly we scale the metric on $T
\times \{i\}$ by the factor $\frac{1}{|i|+1}$. Passing to the limit
in the sequence of metric spaces

$$(T \times
\{0\},|.|) \rightarrow (T \times \{1\},\frac{|.|}{2}) \rightarrow
\cdots \rightarrow (T \times \{i\},\frac{|.|}{i+1}) \rightarrow
\cdots$$

(see \cite{Bestvina, GS1, GJLL, Paulin} for instance) we get a
so-called ``$\alpha$-invariant simplicial $\F{n}$-tree'' ${\mathcal
T}$. This is a simplicial tree, equipped with an isometric action of
$\F{n}$, and which admits moreover an isometry $H$ which
``commutes'' with $\alpha$: $$\mbox{For any } P \in {\mathcal T}
\mbox{, } H(w P) = \alpha(w) H(P).$$

There are non trivial vertex $\F{n}$-stabilizers: these are the
conjugates $w \langle x_1 \rangle w^{-1}$. Thus the tree $\mathcal
T$ is locally infinite. The edge $\F{n}$-stabilizers are trivial.
Although this last property will be found in all the exploited
cases, such a tree is only a very particular case of the $\mr$-trees
below because of its simplicial nature.

\subsection{$\mr$-trees and the observers topology}

A $\mr$-tree $({\mathcal T},d)$ is a geodesic metric space which is
$0$-hyperbolic for the associated distance-function $d$. We often
simply write $\mathcal T$, the distance-function being implicit. The
metric completion of an $\mr$-tree still is an $\mr$-tree.
We denote by
$\overline{{\mathcal T}}$ the union of the completion of ${\mathcal T}$
with its
Gromov boundary $\partial {\mathcal T}$.

\begin{definition} \cite{CHL}
Let $\mathcal T$ be a complete $\mr$-tree.

Let $P,Q$ be any two distinct points in $\overline{{\mathcal T}}$.
A {\em direction at $P$}, denoted by $D_P$,
is a connected component of $\overline{{\mathcal T}} \setminus
\{P\}$. The {\em direction of $Q$ at $P$}, denoted by $D_P(Q)$, is
the connected component of $\overline{{\mathcal T}} \setminus \{P\}$
which contains $Q$.

A {\em branch-point} is any point in $\mathcal T$ at which there are
at least three different directions.

An {\em extremal point} is a
point $P$ at which there is only one direction, i.e.
$\overline{{\mathcal T}} \setminus \{P\}$ is connected. The {\em
interior tree} of $\mathcal T$ is the tree deprived of its extremal
points.

An {\em arc} is a subset isometric to an interval of the real line.
We denote by $[P,Q]$ the geodesic arc from $P$ to $Q$ ($P$ or $Q$,
or both, may belong to $\partial {\mathcal T}$, in which case
$[P,Q]$ denotes the unique infinite, or bi-infinite, geodesic
between $P$ and $Q$).
\end{definition}

Among the $\mr$-trees we distinguish the {\em simplicial
$\mr$-trees}, defined as the $\mr$-trees in which every branch-point $v$
admits a neighborhood $N(v)$ such that each connected component of $N(v) \setminus \{v\}$ is homeomorphic to an open
interval. In these
simplicial trees, it is thus possible to speak of {\em vertices}
(the branch-points) and of {\em edges} (the connected components of
the complement of the vertices). The metric of the $\mr$-tree
defines a length on each edge. If there is a uniform lower-bound on
the length of the edges, the vertices form a discrete subset of the
simplicial $\mr$-tree for the metric topology.

\begin{definition}
Let $\mathcal T$ be a $\mr$-tree. The {\em
observers'topology} on $\overline{{\mathcal T}}$ is the topology
which admits as an open neighborhoods basis the set of all
directions $D_P$, $P \in \overline{{\mathcal T}}$,
and their finite intersections. We denote by
$\widehat{{\mathcal T}}$ the topological space obtained by equipping
$\overline{\mathcal T}$ with the observers'topology.
\end{definition}

The observers topology is weaker than the metric topology
\cite{CHL}. Any sequence of points $(Q_n)$ {\em turning around a
point $P$ of $\mathcal T$}, meaning that every direction at $P$
contains only finitely many of the $Q_n$'s, converges to $P$ in
$\widehat{{\mathcal T}}$. Such a phenomenon is of course only
possible in a non locally finite tree. In a locally finite, simplicial tree,
the metric and observers
topology agree. We summarize below what makes this topology important for
us.

\begin{proposition} \cite{CHL}
\label{merciCHL} Let $\mathcal T$ be a separable
$\mr$-tree. Then:

\begin{enumerate}

  \item $\widehat{{\mathcal T}}$ is compact.

\item The metric topology and the observers'topology agree on
  the Gromov boundary $\partial {\mathcal T}$. In particular
  $\partial {\mathcal T} \subset \widehat{{\mathcal T}}$ is separable.

\item If $(P_n)$ is a sequence of points in $\overline{{\mathcal T}}$ and
$Q$ is a fixed point in $\overline{{\mathcal T}}$, there is a unique
point $P := \underset{n \to \infty}{\liminf_Q} P_n$ defined by

$$\displaystyle [Q,P] = \overline{ \bigcup^{\infty}_{m=0} \bigcap_{n \geq m}
[Q,P_n]}$$

\item If $(P_n)$ is a sequence of point in $\widehat{{\mathcal T}}$
which converges to some point $P$ in $\widehat{{\mathcal T}}$ then
for any $Q \in \overline{{\mathcal T}}$, $P = \underset{n \to
\infty}{\liminf_{Q}} P_n$ holds.

\end{enumerate}
\end{proposition}

\subsection{Affine actions on $\mr$-trees}

\begin{definition}
Let $({\mathcal T},d)$ be a $\mr$-tree. A {\em homothety $H \colon
{\mathcal T} \rightarrow {\mathcal T}$ of dilation factor $\lambda
\in \mr^{+*}$} is a homeomorphism of $\mathcal T$ such that for any
$x,y \in {\mathcal T}$:
$$d(H(x),H(y)) = \lambda d(x,y).$$

The homothety is {\em strict} if $\lambda \neq 1$.
\end{definition}

By the fixed-point theorem, a strict homothety on a complete
$\mr$-tree admits exactly one fixed point.

Let $\mathcal T$ be a $\mr$-tree and let $H$ be a homothety of
$\mathcal T$. Then $H$ induces a permutation on the set of
directions of $\mathcal T$.

\begin{definition} \cite{Liousse}
Let $\mathcal T$ be a $\mr$-tree and let $H$ be a homothety on
$\mathcal T$. An {\em eigenray of $H$ at a point $P$} is a geodesic
ray $R$ starting at $P$ such that $H(R) = R$.
\end{definition}

In particular, if $R$ is an eigenray at $P$, then $H(P) = P$. In
\cite{Liousse} it was proven that if a strict homothety on the
completion of a $\mr$-tree $\mathcal T$ has its fixed point outside
${\mathcal T}$, then it fixes a unique eigenray starting at this
fixed point. In \cite{GJLL} (among others) it was noticed that if
such a homothety leaves invariant a direction $D_P$ ($P$ is thus a
fixed point), then it admits a unique eigenray in $D_P$.

\begin{definition}
\label{affinite} Let $G$ be a discrete group which acts by
homeomorphisms on a $\mr$-tree $({\mathcal T},d)$.

The action of $G$ on $\mathcal T$ is {\em irreducible} if there is no finite
invariant set in $\overline{{\mathcal T}}$.

The action of $G$ on $\mathcal T$ is {\em minimal} if there is no
proper invariant subtree (where a {\em proper} subset of a set is a
subset distinct from the set itself).

The action of $G$ on $\mathcal T$ is an {\em action by homotheties},
or an {\em affine action}, if there is a morphism $\lambda \colon G
\rightarrow \mr^{+*}$ with $d(hx,hy) = \lambda(h) d(x,y)$ for any $h
\in G$ and $x,y \in {\mathcal T}$.

A {\em $G$-tree} is a $\mr$-tree equipped with an isometric action
of $G$.

\end{definition}

We refer to \cite{Liousse} for a specific study of these affine
actions. By the very definition, when given an affine action of a discrete
group $G$ on a $\mr$-tree $\mathcal T$, each element is identified
to a homothety of $\mathcal T$. Thus an affine action of a group $G$
on a $\mr$-tree $\mathcal T$ induces a morphism from $G$ to the
permutations on the set of directions of $\mathcal T$.

A common way in which affine actions arise is as follows:

\begin{definition}
Let $\mathcal T$ be a $G$-tree and let $\theta
\colon {\mathcal U} \rightarrow \Aut{G}$ be a monomorphism.

The $G$-tree $\mathcal T$ is a {\em $\theta(\mathcal U)$-projectively
invariant $G$-tree} if $\theta(\mathcal U)$ acts by homotheties
on $\mathcal T$ such that for any $g \in G$, for any $u \in {\mathcal U}$ and for any point
$P \in {\mathcal T}$:

$$H_{\theta(u)}(g.P) = \theta(u)(g) H_{\theta(u)}(P)$$

If $\theta(\mathcal U)$ acts by isometries then $\mathcal T$ is a {\em
$\theta(\mathcal U)$-invariant $G$-tree}.
\end{definition}

In the expression ``$\theta(\mathcal U)$-projectively invariant
$\mr$-tree'' we will often substitute the group $\theta({\mathcal
U})$ by its generators. In particular, instead of  ``$\langle \alpha
\rangle$-projectively invariant $\mr$-tree'', we will rather write
``$\alpha$-projectively invariant $\mr$-tree''.

\begin{remark}
If $\mathcal T$ is a $\theta(\mathcal U)$-projectively invariant $G$-tree,
then the semi-direct product $G_\theta := G \rtimes_\theta {\mathcal
U}$ acts by homotheties on $\mathcal T$. Conversely, if one has an
affine action of $G_\theta$ on a $\mr$-tree $\mathcal T$ such that
the induced action of $G \lhd G_\theta$ is an isometric action, then
$\mathcal T$ is a $\theta(\mathcal U)$-projectively invariant $G$-tree.
\end{remark}

When one has a $\theta(\mathcal U)$-projectively invariant $G$-tree
$\mathcal T$, then $\mathcal T$ is in fact projectively invariant
for any subgroup $\theta^\prime(\mathcal U) < \Aut{G}$  with
$[\theta(\mathcal U)] = [\theta^\prime(\mathcal U)]$. Indeed, let
$\theta^\prime(u) \in \Aut{G}$ satisfy $\theta^\prime(u)(.) = g_0
\theta(u)(.) g^{-1}_0$ for some
    $g_0
    \in G$. Then $H_{\theta^\prime(u)} = g_0 H_{\theta(u)}$ satisfies
    $H_{\theta^\prime(u)}(g P) = \theta^\prime(u)(g)
    H_{\theta^\prime(u)}(P)$. Indeed

$$H_{\theta^\prime(u)}(g P) = g_0 H_{\theta(u)} (g.P) = g_0
    \theta(u)(g) H_{\theta(u)}(P),$$
$$g_0 \theta(u)(g) H_{\theta(u)}(P) = g_0
    \theta(u)(g) g^{-1}_0 g_0 H_{\theta(u)}(P) = \theta^\prime(u)(g)
    g_0 H_{\theta(u)}(P)$$

and  $g_0 H_{\theta(u)}(P) =  H_{\theta^\prime(u)}(P)$ eventually gives the
announced equality. This justifies to state the

\begin{definition}
Let $\mathcal T$ be a $G$-tree and let $\theta
\colon {\mathcal U} \rightarrow \Out{G}$ be a monomorphism.

The $G$-tree $\mathcal T$ is a {\em $\theta(\mathcal U)$-projectively
invariant $G$-tree} if there is $\tilde{\theta} \colon \mathcal U \rightarrow
\Aut{G}$ with $[\tilde{\theta}(\mathcal U)] = \theta(\mathcal U)$ such that
 $\mathcal T$ is a $\tilde{\theta}(\mathcal U)$-projectively invariant $G$-tree.

\end{definition}

\subsection{The LL-map $\mathcal Q$}

\label{queue}

The LL-map $\mathcal Q$ which appears below was introduced in
\cite{LL1, LL2} in the setting of free group automorphisms (the two
L's of the denomination ``LL-map'' stand for the name of the authors
Levitt-Lustig - see also \cite{CHL} for more details about this
map).

\begin{lemma}
\label{motivons} Let $G$ be a discrete group acting by homotheties
on a $\mr$-tree $\mathcal T$. Let $\overline{G} = G \cup
\partial G$ be a compatible compactification of $G$. Assume that the {\em LL-map
$${\mathcal Q} \colon \left\{
\begin{array}{ccc}
\partial G & \rightarrow & {\mathcal T}\\
X & \mapsto & \underset{n \to \infty}{\liminf_P} g_n P
\end{array} \right.$$}
is well-defined, that is independent from the point $P$ and the
sequence $(g_n)$ converging to $X$ in $\overline{G}$.

Then for any infinite, not eventually constant
     sequence $(g_n)$ of elements in $G$ such that $(g_n  P)$ tends
     to some point $Q$ in $\widehat{{\mathcal T}}$, for
     any $R \in \widehat{{\mathcal T}}$, $(g_n R)$
     tends to $Q$ in $\widehat{{\mathcal T}}$.
\end{lemma}

\begin{proof}
Let $R$ be any point in $\widehat{{\mathcal T}}$ and consider the
sequence $(g_n R)$. By compacity of $\widehat{{\mathcal T}}$ (see
Proposition \ref{merciCHL}), it has at least one accumulation point.
Consider any convergent subsequence $(g_{n_k} R)$ and denote by
$Q^\prime$ its limit. By passing if necessary to a further
subsequence we can suppose that $(g_{n_k})$ tends to some point $X
\in
\partial G$ in $\overline{G}$. By the
assumption that $(g_n P)$ tends toward $Q$, the very definition of
the map $\mathcal Q$ implies ${\mathcal Q}(X) = Q$ . By Proposition
\ref{merciCHL}, $Q^\prime = \underset{n \to \infty}{\liminf_R}
g_{n_k} R$ and since $(g_{n_k})$ converges to $X$, we get $Q^\prime
= {\mathcal Q}(X) = Q$. We so proved that $Q$ is the only
accumulation-point of $(g_n R)$ in $\widehat{{\mathcal T}}$. Since
$\widehat{{\mathcal T}}$ is compact, this implies that $(g_n R)$
tends toward $Q$ in $\widehat{{\mathcal T}}$. The lemma follows.
\end{proof}

 In the case where $G$ is a hyperbolic group
and $\mathcal T$ is a simplicial $G$-tree with quasiconvex vertex
stabilizers and trivial edge-stabilizers, the existence of the
LL-map $\mathcal Q$ is easy to prove. Observe also that the very
existence of the map $\mathcal Q$ implies the triviality of the arc
stabilizers (however, in the cases that we will consider, this
triviality is proved in a direct way rather than by appealing to the
more complex notion that represents the map $\mathcal Q$).
Similarly, for the conclusion of Lemma \ref{motivons} to be true, we
need the triviality of the arc-stabilizers.

\begin{remark}
\label{malnormal}
Consider a $G$-tree $\mathcal T$ with trivial arc-stabilizers.
The stabilizers $H_i$ of the
branch-points are {\em malnormal}: for any $g \in G \setminus H_i$, $g^{-1} H_i g
\cap H_i = \{1\}$. Moreover, any family  $\mathcal H$ of branch-points
stabilizers which belong to distinct
$G$-orbits
form a {\em malnormal family} of subgroups: for any $g \in G$,
for any $H_i \neq H_j$ in $\mathcal H$,
$g^{-1} H_i g \cap H_j = \{1\}$.
\end{remark}

\section{Some classes of groups with affine actions on $\mr$-trees}\label{class}

We gather here various results which provide us with non-trivial
classes of groups with affine actions on $\mr$-trees. Our results
about the Poisson boundary of these groups, developed in the next
sections, will rely upon the theorems recalled here. Once again the
reader can skip for the moment the exposition of these results and
only go back here each time the theorems are referred to.

\subsection{Projectively invariant $\mr$-trees for single automorphisms}

\begin{theorem}

\label{le cas libre}

Let $G$ be the fundamental group of a compact hyperbolic surface $S$
and let $\alpha \in \Aut{G}$ in an outer-class induced by a
homeomorphism of $S$. Assume that $\alpha$ has exponential growth.
Then there exists a separable, minimal $G$-tree ${\mathcal
T}_\alpha$ which is $\alpha$-projectively invariant and satisfies
the following properties:

\begin{enumerate}
 \item The action of $G_\alpha  = G \rtimes_\alpha \mz$ is irreducible (see Definition \ref{affinite}).
 \item The $G$-action has dense orbits in $\overline{{\mathcal
    T}}_\alpha$.
 \item There are a finite number of $G$-orbits
 of branch-points. The stabilizers of branch-points are quasiconvex. For any point
 $P$ in $\mathcal T$, the number
 of $\mathrm{Stab}_G(P)$-orbits of directions at $P$ is finite.
  \item The LL-map $\mathcal Q \colon \partial G \rightarrow \widehat{\mathcal
      T}_\alpha$ is well-defined
(in particular the arc-stabilizers are trivial for the induced
$G$-action), continuous, $G_\alpha$-equivariant and such that the pre-images of
any two distinct points are disjoint compact subsets of $\partial
G$. The pre-image of a point in $\partial {\mathcal T}$ is a single
point of $\partial G$.
\end{enumerate}
\end{theorem}

We refer the reader to \cite{FLP} for the basis about
isotopy-classes of surface homeomorphisms. Since we know no precise
reference for the above theorem, we briefly sketch a proof.

\begin{proof}
Let us first briefly describe how one gets a $\mr$-tree $\mathcal T$
as in Theorem \ref{le cas libre}. We consider a homeomorphism $h$
inducing $[\alpha] \in \Out{G}$ as given by the Nielsen-Thurston
classification: there is a decomposition of $S$ in subsurfaces
$S_1,\cdots,S_r$ such that $h^{k_i}(S_i) = S_i$ and $h^{k_i}_{|S_i}$
is either pseudo-Anosov or has linear growth. The $S_i$'s for which
$h^{k_i}_{|S_i}$ has linear growth are maximal with respect to the
inclusion. Each boundary curve of each $S_i$ is fixed by $h^{k_i}$.
Since $\alpha$ has exponential growth, there is at least one
subsurface $S_i$ such that $h^{k_i}_{|S_i}$ is pseudo-Anosov. Let
$S_0$ be the subsurface for which $h^{k_0}_{|S_0}$ is pseudo-Anosov
with greatest dilatation factor, denoted by $\lambda_0$. We now
consider the universal covering $\tilde{S}$ of $S$ (this is a
hyperbolic Poincar\'e disc $\md^2$ in the closed case, the
complement of horoballs in $\md^2$ in the free group case) and a
lift $\tilde{h}$ of $h$. Each connected component of $\tilde{S}_i$,
$i > 0$, each reduction curve and each boundary curve is collapsed
to a point. Since the corresponding subgroups are quasi convex and
malnormal, the resulting metric space $\hat{S}$ is still a Gromov
hyperbolic space (of course not proper). The connected components of
$\tilde{S}_0$ are equipped with a pair of invariant, transverse,
transversely measured singular foliations $({\mathcal F}_u,\mu_s)$
and $({\mathcal F}_s,\mu_u)$: $\tilde{h}(({\mathcal F}_u,\mu_s)) =
({\mathcal F}_u,\frac{1}{\lambda_0} \mu_s)$ and
$\tilde{h}(({\mathcal F}_s,\mu_u)) = ({\mathcal F}_s,\lambda_0
\mu_u)$. We equip $\hat{S}$ with the length-metric associated to the
transverse measures: $|\gamma|_{\mathcal F}= \int_\gamma |d\mu_u| +
|d\mu_s|$. The resulting metric space $\hat{S}_{\mathcal F}$ is
still Gromov hyperbolic. We now consider the limit-metric space
$$(\hat{S}_{\mathcal F} \times \{0\},|.|_{\mathcal F}) \rightarrow
(\hat{S}_{\mathcal F} \times \{1\},\frac{|.|_{\mathcal
F}}{\lambda_0}) \rightarrow \cdots \rightarrow (\hat{S}_{\mathcal F}
\times \{i\},\frac{|.|_{\mathcal F}}{\lambda^i_0}) \rightarrow
\cdots$$ obtained by rescaling the metric $|.|_{\mathcal F}$ by
$\lambda^i_0$ for $i \to + \infty$. From \cite{GJLL}, the
limit-space is a $\alpha$-projectively invariant $G$-tree $\mathcal
T$ such that the $G$-action has trivial arc-stabilizers (\cite{GJLL}
deals with free group automorphisms but, as noticed there, the
triviality of the arc-stabilizers is proved in the same way when
considering a hyperbolic group). Possibly after restricting to a
minimal subtree, the $G$-action has dense orbits, see for instance
\cite{Paulin}: this is true as soon as the action of $\alpha$ is by
strict
homotheties and the $\mr$-tree is minimal.\\

{\em Claim:} Let ${\mathcal L}_s$ be the geodesic lamination
associated to ${\mathcal F}_s$. Each non-boundary leaf of ${\mathcal
L}_s$ is mapped to a point in $\mathcal T$ which is not a
branch-point. Any two points in $\hat{S}_{\mathcal F}$ on distinct
leaves of ${\mathcal L}_s$ are mapped to distinct points in
$\mathcal T$. The $\mr$-tree $\mathcal T$ is equivalently obtained
by collapsing each leaf of ${\mathcal F}_s$ to a point. A
stabilizers of a branch-point is either conjugate to a subgroup
associated to a subsurface $S_i$, $i > 0$, or is a cyclic subgroup
corresponding to a boundary curve of $S$ or to a reduction curve.

{\em Proof:} The first two assertions are
an easy consequence of the fact that:

\begin{itemize}
  \item $\tilde{h}$ is contracting by $\lambda_0$
along the leaves of ${\mathcal L}_s$,
  \item given any path $p$ transverse to ${\mathcal F}_s$, some iterate of $\tilde{h}$
  dilates the length of $p$
by $C \lambda_0$, for some uniform constant $C$.
\end{itemize}

The last two assertions of the claim are clear.\\

Let us check the fourth assertion of the theorem. First we map the
boundary of each conjugate of each subgroup associated to a
subsurface ${S}_i$ to the branch-point stabilized by this subgroup.
Second, since the BBT-property of \cite{GJLL} is satisfied, there is
a well-defined $G$-equivariant map from a subset of $\partial G$
onto $\partial \mathcal T$ which is one-to-one. Thus the LL-map
$\mathcal Q$ is well-defined on points $X \in \partial G$ which are
limits of sequences $(g_n) \subset G$ whose translation-lengths on
$\mathcal T$ tend to infinity. Choose a point $P \in
\overline{\mathcal T}$ and let us now consider a sequence $(g_n P)$
with $(g_n) \to X \in \partial G$ such that the translation-lengths
of the $g_n$'s do not tend toward infinity. Then, it is again a
consequence of the BBT property that $\displaystyle \liminf_{n \to
\infty} g_n P$ does not depend on the sequence $(g_n)$ tending
toward $X$. Let us prove that it is independent from $P$. It
suffices to check it for a sequence $(g_n)$ the translation-lengths
of which, denoted by $(t_n)$, tend to $0$. If $(t_n)$ is eventually
constant then for some $N$ all the $g_k$'s with $k \in N$ belong to
the stabilizer of some branch-point $B$. Since arc stabilizers are
trivial, the $g_k P$ turn around $B$, which is then equal to
$\displaystyle \liminf_{n \to \infty} g_n P$. The same is true if
the $g_n$'s tend to some point in the boundary of the stabilizer of
a branch-point. In the remaining case, the $g_n$'s correspond to
elements of $G$ which are realized by a sequence of paths $(p_n)$
closer and closer to a given leaf of ${\mathcal L}_s$. The terminal
points of the $p_n$'s belong to different leaves of ${\mathcal
L}_s$. Hence this is true for the paths $q_n$ associated to a
sequence $g_n q$, where $q$ is a fixed path from the base-point to a
leaf in ${\mathcal L}_s$. Since different leaves of ${\mathcal L}_s$
correspond to different points of $\mathcal T$, the conclusion
follows.

The continuity of the map $\mathcal Q$ is proved as in \cite{CHL}.
\end{proof}

\begin{remark}
\label{blabla} Assume that $\alpha$ is induced by a pseudo-Anosov
homeomorphism $h$ of a closed compact hyperbolic surface $S$. Let
$x,y$ be the endpoints in $\partial \mh^2$ of the lift of a stable
or unstable leaf of $h$. Then there are no geodesics in $G_\alpha$
from $x$ to $y$: indeed the exponential contraction of the leaf
implies that it can be exhausted by compact intervals $I_n$ which
are exponentially contracted by $\lambda^{-t}$ for $t \to \infty$;
this forbids the function $2t+\lambda^{t} |I_n|$ to attain a minimum
for $|I_n| \to \infty$. In fact, the group $G_\alpha$ is hyperbolic
and these two points of $\partial \mh^2$ get identified to a single
one in the Gromov compactification. In this case, ${\mathcal Q}(x)$
and ${\mathcal Q}(y)$ are the endpoints of a same eigenray, or the
endpoints of two eigenrays with a same origin.
\end{remark}

\begin{definition}
\label{fermer les eigenrays}
Let $G$ be a discrete group acting by homotheties on an $\mr$-tree
$\mathcal T$. If $R$ is an eigenray of some homothety $H_g$ of
$\mathcal T$, then {\em closing $R$} means identifying its two
endpoints in $\widehat{\mathcal T}$.
\end{definition}

\begin{proposition}

\label{importante} With the assumptions and notations of Theorem
\ref{le cas libre}, let $\widetilde{\mathcal T}$ be the space
obtained by closing all the eigenrays in $\widehat{\mathcal T}$ and
let $q \colon \widehat{\mathcal T} \rightarrow \widetilde{\mathcal
T}$ be the associated quotient-map. Then:

\begin{enumerate}
  \item $\widetilde{\mathcal T}$ is Hausdorff and compact.

 \item The action of $G_\alpha$ on $\widehat{\mathcal T}$ descend, via $q$, to
an irreducible action by homeomorphisms on $\widetilde{\mathcal T}$.
The induced $G$-action has dense orbits and restricts to an
isometric action on the completion of ${\mathcal T} \subset
{\widetilde{\mathcal T}}$. The $G_\alpha$-action restricts to an
affine action on the completion of $\mathcal T \subset
\widetilde{\mathcal T}$.

 \item If $X, Y$ are two distinct points in $\partial G$ which cannot be
   connected
by a bi-infinite $G_\alpha$-geodesic, then $q({\mathcal Q}(X)) = q({\mathcal Q}(Y))$.

\end{enumerate}

\end{proposition}

This result is certainly not a surprise to experts in the field of
surface or free group automorphisms. In the case of mapping-tori of
pseudo-Anosov surface homeomorphisms, it essentially amounts to the
knowledge of the famous Cannon-Thurston map and of what are its
fibers. We refer the reader to \cite{CT, Mitra} for this point of
view. It seems however that the above statement has not been written
under this form and under this generality, so that we give a brief
proof for completeness.

\begin{proof} The notations are those introduced in the proof of
Theorem \ref{le cas libre}. Let $l$ be a leaf of the unstable
geodesic lamination ${\mathcal L}_u$ in $\hat{S}_{\mathcal F}$.
Either it connects two points in $\partial \hat{S}_{\mathcal F}$, or
it connects a point in $\partial \hat{S}_{\mathcal F}$ to the
collapse of a lift of a reduction curve, a boundary curve or a
subsurface $S_i$, $i > 0$. We now denote by $l$ a leaf as above or
the subset of a singular leaf of ${\mathcal F}_u$, homeomorphic to
$\mr$, the endpoints of which are of the above form. Such a subset
is a {\em bad subset}. Then the endpoints of $l$ are not connected
by a $G_\alpha$-geodesic: the argument given in Remark \ref{blabla}
in the pseudo-Anosov case applies here.

Conversely, any two points in $\hat{S}_{\mathcal F} \cup \partial
\hat{S}_{\mathcal F}$ which are not the endpoints of a bad subset
$l$ as defined above are connected by a $G_\alpha$ geodesic. Indeed,
they are connected by a path $p$ which can be put transverse to the
two foliations so that both its stable and unstable lengths (i.e.
the lengths measured by integrating the stable and unstable measures
of the foliations) are positive. It follows that sufficiently long
subpaths of $p$ are shortened under iterations of $\tilde{h}$ or
$\tilde{h}^{-1}$ until reaching a positive minimal length $L$. The
computations carried on in \cite{GauteroLustigvieux} imply that a
$G_\alpha$-geodesic passes in a neighborhood of the path with length
$L$ (the size of the neighbrohood depends on $L$ which has been
chosen sufficently large enough once and for all).

An eigenray $R$, with origin $O$ and endpoint $w \in \partial
\mathcal T$ satisfies $wH(R)=R$. Since $w$ acts by isometries,
whereas $H$ is a homothety with dilatation factor $\lambda_0 > 1$,
this implies that the length of any subpath in a geodesic $g$ in
$\hat{S}_{\mathcal F}$ from ${\mathcal Q}^{-1}(w)$ to the boundary
of the convex-hull of ${\mathcal Q}^{-1}(O)$ is exponentially
contracted under iterations of $\tilde{h}^{-1}$. Therefore the
endpoints of $g$ cannot be connected by a $G_\alpha$-geodesic. It
follows that its endpoints are the endpoints of a bad subset $l$ of
$\hat{S}_{\mathcal F} \cup \partial \hat{S}_{\mathcal F}$.
Conversely it is clear that the endpoints of a bad subset $l$
project, under $\mathcal Q$, to the endpoints of an eigenray or to
the endpoints of two eigenrays with a same origin.

Now the endpoints of the bad subsets form a closed set in the
following sense: if $(X_n,Y_n) \in \partial G \times \partial G$ are
endpoints of bad subsets tending toward $(X,Y)$ then $X$ and $Y$ are
the endpoints of a bad subset. It follows that the collection of
eigenrays satisfy a similar property: if $R_n$ is a sequence of
eigenrays with origin $O_n$ and terminal point $w_n \in \partial
\mathcal T$ with $O_n$ (resp. $w_n$) tending toward some $O$ (resp.
$w$) in $\widehat{\mathcal T}$, then either $O$ and $w$ are the
endpoints of an eigenray, or they are the endpoints of two eigenrays
with same origin. It readily follows that closing the eigenrays
gives a Hausdorff space. The compacity is then easily deduced from
the compacity of $\widehat{\mathcal T}$.

We so got
the first and third items of the proposition. The second item follows from the fact
that $G_\alpha$ permutes the eigenrays.
\end{proof}

There are ``generalizations'' of these theorem and proposition
toward free groups and torsion free hyperbolic groups with
infinitely many ends. The paper \cite{GJLL} gives the conclusions of
Theorem \ref{le cas libre} in the setting of free group
automorphisms and the point (2) of Theorem \ref{le cas hyperbolique}
below dealing with polynomially growing automorphisms. The paper
\cite{LL2}[Theorems 10.4 and 10.5] allows us to generalize
\cite{GJLL} to torsion free hyperbolic groups with infinitely many
ends. Let us recall that such a hyperbolic group $G$ is the
fundamental group of a graph of groups with trivial edge stabilizers
and one-ended hyperbolic groups as vertex stabilizers. The {\em
relative length} of an element $\gamma$ of $G$ is the word-length
associated to the (infinite) generating set obtained by adding every
element of any vertex stabilizer to the given generating set. An
automorphism $\alpha$ of a torsion free hyperbolic group with
infinitely many ends $G$ has an {\em essential exponential growth}
if there exists an element $\gamma$ of $G$ such that the relative
length of $\alpha^j(\gamma)$ growths exponentially with $j \to +
\infty$. Otherwise $\alpha$ has {\em essential polynomial growth}. Indeed, any automorphism
of a torsion free hyperbolic group $G$ with infinitely many ends has either essential exponential or essential
polynomial growth: this is a straightforward consequence of the fact that any free group automorphism has either
exponential or polynomial growth, the free group involved here being the quotient of $G$ by the vertex stabilizers 
mentioned above and their conjugates.

\begin{theorem}[\cite{Gaboriau-Levitt, GJLL, LevittLustig, LL1, LL2, CHL}]
\label{le cas hyperbolique} Let $G$ be a torsion
free hyperbolic group with infinitely many ends (in particular $G$ may be a non-abelian free group).

\begin{enumerate}
  \item If $\alpha$ is an automorphism
with essential exponential growth then the conclusions of Theorem \ref{le cas libre} and Proposition
\ref{importante} are true.
  \item If $\alpha$ is an automorphism with essential polynomial growth
  then there is a minimal, simplicial $\alpha$-invariant $G$-tree
$\mathcal T$ which satisfies all the properties of Theorem \ref{le
cas libre}.
\end{enumerate}
\end{theorem}

\begin{proof}[Proof of Proposition \ref{importante} in the setting
of Theorem \ref{le cas hyperbolique}]

We consider the Cayley graph $\Gamma$ of $G$ with respect to some
finite generating set. We consider a tree $\mathcal T$ as given by
Theorem \ref{le cas hyperbolique}. Up to conjugacy in $G$, there are
a finite number of stabilizers of branch-points ${\mathcal H} =
\{H_1,\cdots,H_r\}$. The family $\mathcal H$ is quasi convex, and
malnormal. We consider the coned-off Cayley graph $\Gamma_{\mathcal
H}$ (see Section \ref{rh}). We consider a cellular, piecewise linear
map $f \colon \Gamma_{\mathcal H} \rightarrow \Gamma_{\mathcal H}$
which realizes the given automorphism. Since the subgroups $H_i$'s
are preserved up to conjugacy, $f$ can be assumed to permutes the
vertices of the cones in $\Gamma_{\mathcal H}$. We denote by
$\Gamma^f_{\mathcal H}$ the {\em mapping-telescope} of $(\Gamma_{\mathcal
H},f)$, that is the disjoint union of the $K_i := \Gamma_{\mathcal
H} \times [0,1]$ quotiented by the equivalence relation $(x,1) \in K_i \sim (f(x),0) \in K_{i+1}$.
A {\em horizontal geodesic} in $\Gamma^f_{\mathcal H}$ is
a $\Gamma_{\mathcal H}$-geodesic contained in some {\em stratum}
$\Gamma_{\mathcal H} \times \{j\}$, $j \in \mz$. It is {\em simple}
if it contains no vertex of any cone in $\Gamma_{\mathcal H}$.

\begin{definition}
\label{corridor}
A {\em corridor} is a union of (possibly infinite or bi-infinite)
simple horizontal geodesics, exactly one in each stratum, which
connect two orbits of $f$.
\end{definition}

A {\em $K$-quasi orbit} is a sequence of points
$x_0,\cdots,x_j,\cdots$  such that there is a sequence of vertical
segments $v_1,\cdots,v_{j+1},\cdots$ of lengths at least one
satisfying that the initial point of $v_i$ is $x_{i-1}$ and the
terminal point of $v_i$ lies at horizontal distance at most $K$ from
$x_i$. By \cite{Gautero3}, there is a constant $K$ such that, if $x$
is a point in a corridor $\mathcal C$, then there is a $K$-quasi
orbit passing through $x$ and contained in $\mathcal C$.

\begin{definition}
A corridor $\mathcal C$ is {\em collapsed toward $+ \infty$} (resp.
{\em toward $- \infty$}) if, given any two points $x,y$ in $\mathcal
C \cap (\Gamma_{\mathcal H} \times \{j\})$, there exist $l \in \mn$
and two $K$-quasi orbits starting respectively at $x$ and $y$ and
ending at the same point in $\mathcal C \cap (\Gamma_{\mathcal H}
\times \{j+l\})$ (resp. $\mathcal C \cap (\Gamma_{\mathcal H} \times
\{j-l\})$).

A {\em collapsing corridor} is a corridor which collapses either
toward $+ \infty$ or toward $- \infty$.
\end{definition}

The collapsing corridors play the r\^ole of the stable and unstable
leaves in the setting of Theorem \ref{le cas libre}.

\begin{lemma}
\label{il faut le rajouter} The subset of the collapsing corridors
is closed in the following sens: if $\mathcal C$ is a non-collapsing
corridor, then there exists $L > 0$ such that the horizontal
$L$-neighborhood of $\mathcal C$ contains only non-collapsing
corridors.
\end{lemma}

\begin{corollary}
Let $(X_n,Y_n) \subset \partial \Gamma_{\mathcal H} \times \partial
\Gamma_{\mathcal H}$ be a sequence of points which are the endpoints
of a sequence of collapsing corridors. Assume that $X_n \to X$ and
$Y_n \to Y$. Then $X$ and $Y$ are the endpoints of a collapsing
corridor.
\end{corollary}

As in the setting of Theorem \ref{le cas libre}, two points in
$\partial \Gamma_{\mathcal H}$ are connected by a
$G_\alpha$-geodesic if and only if they define a non-collapsing
corridor. The correspondance between the collapsing corridors and
the eigenrays is established as was previously done. The proof of
the proposition is then completed in the same way. We leave the
reader work out the details.
\end{proof}

\subsection{Subgroups of polynomially growing free group automorphisms}

\begin{theorem} \cite{BFH}
\label{Bestvina2} Let $\mathcal P$ be a finitely generated group and
let $\theta \colon {\mathcal P} \rightarrow \Out{\F{n}}$ be a
monomorphism such that $\theta({\mathcal P})$ consists entirely of
polynomially growing automorphisms. Then there is a finite-index
subgroup $\mathcal U$ of $\mathcal P$, termed {\em unipotent
subgroup}, and a simplicial $\tilde{\theta}(\mathcal U)$-invariant
$\F{n}$-tree $\mathcal T$ with $[\tilde{\theta}(\mathcal U)] =
\theta(\mathcal U)$ which satisfies the following properties:

\begin{enumerate}
  \item A vertex of $\mathcal T$ is fixed by all the isometries
$H_{\tilde{\theta}(u)}$, $u \in {\mathcal U}$, and this is the
unique fixed point of each one.
  \item The $\F{n}$-action has trivial edge stabilizers.
  \item There is exactly one $\F{n}$-orbit of edges.
  \item For each vertex $v$ of $\mathcal T$, each
  $\mathrm{Stab}_{\F{n}}(v)$-orbit of directions at $v$ is invariant under
  the $\mathcal U$-action.

\end{enumerate}
\end{theorem}

As was already observed, the map $\mathcal Q$ is well-defined here
since vertex stabilizers are quasiconvex and arc stabilizers are
trivial.

\section{Cyclic extensions over exponentially growing automorphisms}\label{exp}

The bulk of this section is to prove Theorem \ref{resultat 1} below
about cyclic extensions of surface groups over exponentially growing
automorphisms. Then we will briefly show how the same methods apply
to similar cyclic extensions of free groups and of torsion free
hyperbolic groups with infinitely many ends, see Theorem
\ref{resultat 3}.

\subsection{The surface case}

\begin{theorem}
\label{resultat 1} Let $G$ be the fundamental group of a compact,
hyperbolic surface $S$ (with or without boundary). Let $\alpha \in
\Aut{G}$ be an exponentially growing automorphism in an outer-class
induced by a homeomorphism of $S$. Let $\mu$ be a probability
measure on $G_\alpha := G \rtimes_\alpha \mz$ whose support
generates $G_\alpha$ as a semi-group.

Then there exists an $\alpha$-projectively invariant $G$-tree
$\widehat{\mathcal T}$ with minimal interior such that, if
$\widetilde{\mathcal T}$ denotes the space obtained by closing all
the eigenrays of $\widehat{{\mathcal T}}$, then:

\begin{enumerate}
  \item {\bf P}-almost every sample path ${\bf x} = \{x_n\}$
  converges to some $x_\infty \in \widetilde{\mathcal T}$.
  \item The hitting measure $\lambda$, which
  is the distribution of $x_\infty$, is a non-atomic measure on
  $\widetilde{{\mathcal T}}$ such that $(\widetilde{{\mathcal
  T}},\lambda)$ is a $\mu$-boundary of $(G_\alpha,\mu)$ and $\lambda$ is the unique
  $\mu$-stationary probability measure on $\widetilde{{\mathcal
  T}}$.
  \item If the measure $\mu$ has finite first
  logarithmic moment and finite entropy with respect to a word-metric
  on $G_\alpha$, then the measured space $(\widetilde{{\mathcal
  T}},\lambda)$ is the Poisson boundary of $(G_\alpha,\mu)$.
\end{enumerate}
\end{theorem}

We consider a $\mr$-tree $\mathcal T$ as given by Theorem \ref{le
cas libre}. We denote by $t$ the generator of $\mz$ acting on the
right as $\alpha$ on $G$. Beware that $t$ acts on the left as the
homothety $H^{-1}_\alpha$ on $\widehat{\mathcal T}$.

Choose a point $O$ as base-point and
identify $G_\alpha/\mathrm{Stab}_{G_\alpha}(O)$ with a subset of the points of $\mathcal T$ by
considering the orbit $G_\alpha . O$ of the base-point. Of course, in
the case where $O$ has a non-trivial stabilizer, infinitely
many elements of $G_\alpha$ may get identified with a single point. {\em For the clarity and briefness of some arguments we choose as base-point the fixed-point of the homothety $H$}.
The following lemma is obvious:

\begin{lemma}
\label{obvious}

If $P$ is any point in $\widehat{{\mathcal T}}$, let ${\mathcal
N}_{\widehat{\mathcal T}}(P)$ consist of a neighborhood ${\mathcal
N}^{obs}(P)$ of $P$ in $\widehat{\mathcal T}$ together with all the
elements $g u \in G_\alpha$ such that $gu O \in {\mathcal
N}^{obs}(P)$ and $|gu|
> N$ for some chosen positive $N$.

Then $G_\alpha \cup \widehat{\mathcal T}$, equipped with the above basis of
neighborhoods ${\mathcal N}_{\widehat{\mathcal T}}(.)$, is a separable, compatible
compactification of $G_\alpha$.

The action of $G \lhd G_\alpha$ on $\widehat{\mathcal T}$ is the
given isometric left action.

The action of $\mz < G_\alpha$ on $\widehat{{\mathcal T}}$
is given by $u . P = H_{u^{-1}}(P)$.
\end{lemma}

\begin{remark}
If there exists a point $O$ in $\mathcal T$ with trivial $G_\alpha$-stabilizer we have an identification of $G_\alpha$ with the orbit $G_\alpha.O$. In the cases considered in this paper it is always possible to choose such a point $O$ when dealing with exponentially growing automorphisms. This comes from the fact that there are a finite number of orbits of branch-points and of course a countable number of points in a given orbit. Thus there are a countable number of points with non-trivial stabilizer whereas the tree has an uncountable number of points. Thus, carefully choosing the base-point would allow us to get rid of the ``size'' given by the integer $N$ in the definition of the compactification of Lemma \ref{obvious}.

When dealing later with polynomially growing automorphisms, the trees are simplicial and, since edges have trivial $G$-stabilizers, it suffices to choose the middle of an edge as base-point to get a point with trivial $G$-stabilizer. However in this simplicial case it might happen that, whatever point is chosen, it has a non-trivial stabilizer in $G_\alpha$. Nevertheless, as the reader shall see, in this case we consider a product of two trees and once again for the given $G_\alpha$-action it is possible to choose a point with trivial $G_\alpha$-stabilizer. 
\end{remark}

Unfortunately, the
compactification above will not satisfy the (CP) condition. We now need to close
the eigenrays as in Theorem \ref{resultat 1}.

\begin{proposition}
\label{a faire} In the compactification of $G_\alpha$ by
$\widehat{\mathcal T}$ given in Lemma \ref{obvious}, extend the
quotient-map $q \colon \widehat{\mathcal T} \rightarrow
\widetilde{\mathcal T}$ by the identity on $G_\alpha$.

Define ${\mathcal N}_{\widetilde{\mathcal T}}(P) =
q({\mathcal N}_{\widehat{\mathcal T}}(q^{-1}(P)))$.

Then $G_\alpha \cup \widetilde{\mathcal T}$, equipped with the basis of neighborhoods
${\mathcal N}_{\widetilde{\mathcal T}}(.)$
is a separable, compatible compactification of $G_\alpha$. The action of
$G_\alpha$ on $\widetilde{\mathcal T}$ is irreducible.
\end{proposition}

\begin{proof}
By Proposition \ref{importante}, $\widetilde{\mathcal T}$ is
Hausdorff and compact. Since the quotient-map $q$ restricts to the
identity on $\mathcal T$, the neighborhood in $G_\alpha$ of a point
of $\widetilde{\mathcal T}$ is the same as the neighborhood of this
point in $\widehat{\mathcal T}$. It is then obvious that Proposition
\ref{a faire} gives a compatible compactification of $G_\alpha$ by
$\widetilde{\mathcal T}$. The irreducibility of the action comes from
Proposition \ref{importante}.
\end{proof}

Let us see what happens with respect to the (CP) and (CS)
conditions when considering this compactification.

\begin{proposition}
\label{verification de la condition CP} The compactification of
$G_\alpha$ with $\widetilde{\mathcal T}$ given by Proposition
\ref{a faire} satisfies Kaimanovich (CP) condition.

\end{proposition}

\begin{proof}
Consider any sequence $(w_j) \in G$ which tends to some point $P \in
\widehat{\mathcal T}$, which means that $(w_jO)$ tends to $P$ in
$\widehat{\mathcal T}$. Let $v \in G \lhd G_\alpha$. By Theorem
\ref{le cas libre}, the LL-map $\mathcal Q$ is well-defined. By
Lemma \ref{motivons}, $(w_j.(vO)) = (w_jvO)$ tends to the same point
$P \in \widehat{{\mathcal T}}$. The condition (CP) is thus satisfied
for $\widehat{\mathcal T}$, and so for $\widetilde{\mathcal T}$, in
this case. If $v = t^k$ for some integer $k$ the same conclusion
holds trivially from the very definition of the compactification:
the elements $w$ and $wt^k$ lie in a neighborhood of the same point
in the tree since the base-point $O$ has been chosen as the
fixed-point of $H$, that is of the action of $t$ on
$\widehat{\mathcal T}$.

Let us now consider a sequence $t^{n_j}$
with $n_j \to + \infty$. Of course $t^{n_j} O$
tends to $O$ in $\widehat{\mathcal T}$ so that
$t^{n_j}$ tends to $O$ in the compactification
with $\widetilde{\mathcal T}$. Let $v \in G_\alpha$.
If $v = t^k$ for some integer $k$ then $t^{n_j}v = t^{n_j+k}$
tends to $O$ in $\widehat{\mathcal T}$ so that the (CP) condition still holds.
If $v \in G$, $t^k v O$ admits $O$ and the endpoints of the
eigenrays of $H$
as accumulation points.
Thus it converges in $\widetilde{\mathcal T}$
to the point $O$, since the eigenrays have been closed. Therefore the (CP) condition holds in
$\widetilde{\mathcal T}$ in this case.

The lemma readily follows.
\end{proof}

\begin{remark}
In the last case treated in the proof of Proposition \ref{verification de la
  condition CP}, we really need to be in $\widetilde{\mathcal T}$, and not
only in $\widehat{\mathcal T}$.
\end{remark}

We are now going to deal with the (CS) condition. We need a preliminary definition:

\begin{definition}
Assume that an identification of $\mathrm{Stab}_{G}(O)$ with the directions at the base-point $O$ has been chosen once and for all.
Let $P \in \widehat{\mathcal T}$ be a point with non-trivial $G$-stabilizer $H < G$. If $D_P(Q)$
is the direction of $Q$ at $P$, the {\em exit-point of $H$ in $D_P(Q)$} is the element $g \in H$ such that $D_P(g.O) = D_P(Q)$.
\end{definition}

Before giving the definition of the strips, we recall that the mapping-torus group $G_\alpha$ of an exponentially
growing automorphism of a surface or free group $G$ (or of an essentially exponentially growing automorphism of a torsion free hyperbolic group with infinitely many ends) admits a non-trivial structure of
strongly relatively hyperbolic group (see Section \ref{rh}): the $G_\alpha$-geodesics we evoke below are geodesics for the corresponding so-called ``relative'' metric.

\begin{definition}
\label{ca avance} We identify $G$ and $\partial G$ with $G \times \{0\}$ and $\partial G \times
\{0\}$.  If $b_1 \neq b_2$ are two distinct points in $\widehat{\mathcal T}$, we denote by
${\mathcal G}(b_1,b_2)$ the union of all the $G_\alpha$-geodesics
between

\begin{itemize}
  \item the exit-point of $\mathrm{Stab}_G(b_1)$ in $D_{b_1}(b_2)$ and the exit-point of $\mathrm{Stab}_G(b_2)$ in $D_{b_2}(b_1)$ if both $G$-stabilizers are non-trivial;
  \item the exit-point of $\mathrm{Stab}_G(b_1)$ (resp. of $\mathrm{Stab}_G(b_2)$) in $D_{b_1}(b_2)$ (resp. in $D_{b_2}(b_1)$)
  and the point ${\mathcal Q}^{-1}(b_2) \in \partial G$ (resp. and the point ${\mathcal Q}^{-1}(b_1)) \in \partial G$ if the $G$-stabilizer of $b_1$ (resp. of $b_2$) is non-trivial whereas the $G$-stabilizer of
  $b_2$ (resp. of $b_1$) is;
  \item the point ${\mathcal Q}^{-1}(b_1) \in \partial G$ and the point ${\mathcal Q}^{-1}(b_2) \in \partial G$ if both $G$-stabilizers are trivial,
 \end{itemize}
 
 where $\mathcal Q$ is the LL-map defined in Lemma \ref{motivons}.

Let $\Delta$ be the diagonal of $\widehat{\mathcal T} \times
\widehat{\mathcal T}$. Let $\mathcal R$ be a set of representatives
of $G_\alpha$-orbits in $(\widehat{\mathcal T} \times
\widehat{\mathcal T} \setminus \Delta) / ((x,y) \sim (y,x))$.

If $(b^\prime_1,b^\prime_2) \in \widehat{\mathcal T} \times
\widehat{\mathcal T} \setminus \Delta$, the {\em elementary strip
$ES(b^\prime_1,b^\prime_2)$ between $b^\prime_1$ and $b^\prime_2$}
is the union of all the $gu.{\mathcal G}(b^{\prime \prime}_1,b^{\prime
\prime}_2)$ with $(b^{\prime \prime}_1,b^{\prime \prime}_2) \in {\mathcal R}$ and $gu \{b^{\prime \prime}_1,b^{\prime \prime}_2\} =
\{b^\prime_1,b^\prime_2\}$.

Let $b_1 \neq b_2$ be two distinct points in $\widetilde{\mathcal T} \times \widetilde{\mathcal
T} \setminus \Delta$. The {\em strip} $S(b_1,b_2)$ is defined by

$S(b_1,b_2) = ES(b^\prime_1,b^\prime_2)$ with $b^\prime_i = b_i$ if
 $q^{-1}(b_i) = b_i$ and $b^\prime_i$ is the unique origin of
 eigenray in $q^{-1}(b_i)$ otherwise, where $q \colon \widehat{\mathcal T} \rightarrow \widetilde{\mathcal
T}$ is the quotient-map which ``closes the eigenrays'' (see Definition \ref{fermer les eigenrays} and Proposition \ref{importante}).
\end{definition}

\begin{lemma}
\label{vide}
No strip is empty.
\end{lemma}

\begin{proof}
Let $b_1, b_2$ be two distinct points in $\widetilde{\mathcal T}$.
We denote by $b^\prime_i$ the points in $\widehat{\mathcal T}$ with
$q(b^\prime_i) = b_i$ given in Definition \ref{ca avance} for
defining $S(b_1,b_2)$. By Proposition \ref{importante} $\mathcal G(b^\prime_1,b^\prime_2)$ is empty if and only if $b^\prime_1$
and $b^\prime_2$ are the endpoints of an eigenray.  Since the eigenrays
have been closed in $\widetilde{\mathcal T}$, we would get $b_1 =
b_2$ which is a contradiction with the assumption. Hence the lemma.
\end{proof}

\begin{lemma}
\label{yes we can} The map $$\left\{
\begin{array}{ccccc} \widetilde{\mathcal T} &
\times & \widetilde{\mathcal T} & \rightarrow & G_\alpha \\
(b_1 & , & b_2) & \mapsto & S(b_1,b_2) \end{array} \right.$$ is
$G_\alpha$-equivariant and Borel.
\end{lemma}

\begin{proof}
The $G_\alpha$-equivariance is clear by construction. We only have
to check that the map which assigns the elementary sets is Borel.
This is straightforward since two distinct points in
$\widetilde{\mathcal T} \times \widetilde{\mathcal T} \setminus
\Delta$ have distinct image sets. Therefore
any set in $G_\alpha$, which is countable, is a countable union of
points in $\widetilde{\mathcal T} \times \widetilde{\mathcal T}
\setminus \Delta_{\widetilde{\mathcal T}}$, hence is Borel.
\end{proof}

\begin{proposition}
\label{verification de la condition CS}

The $G_\alpha$-compactification $\widetilde{\mathcal T}$, equipped with the collection
of strips $S(b_1,b_2)$, satisfies
the (CS) condition.

\end{proposition}

\begin{proof}
From Lemma \ref{vide}, no strip is empty. From Lemma \ref{yes we
can}, the assignment of the strips is Borel and
$G_\alpha$-equivariant. It only remains to check that any strip
$S(b_1,b_2)$ avoids a neighborhood of any third point $b_0 \in
\widetilde{\mathcal T}$ distinct from both $b_1$ and $b_2$. If no
element of $G_\alpha$ fixes both $b_1$ and $b_2$, then $S(b_1,b_2)$
consists of a $\delta$-thin pencil of $G_\alpha$-geodesics between two points in a Gromov hyperbolic space. The
only accumulation-points of these geodesics are by definition the
boundary-points. If some element $wt^j$ ($j \neq 0$) would fix both $b_1$ and
$b_2$, they would be the endpoints of an eigenray. Since eigenrays
have been closed in $\widetilde{\mathcal T}$, if $b_1$ and $b_2$ are
fixed by a same element of $G_\alpha$ this is an element $w$ in $G$
which acts as a hyperbolic isometry whose axis admits $b_1$ and
$b_2$ as endpoint. It readily follows that these are the only
accumulation-points.
\end{proof}

\begin{proposition}
\label{Francois6.5} Let ${\mathcal J}$ be a finite word gauge for
$G_\alpha$. Then the strips $S(b_1,b_2)$ given by Proposition
\ref{verification de la condition CS} grow polynomially with respect
to $\mathcal J$.
\end{proposition}

\begin{proof}
The group $G_\alpha$ is strongly hyperbolic relative to the mapping-tori
of the maximal subgroups where the automorphism $\alpha$ has linear growth
(see Section \ref{rh}, Theorem \ref{lustigmoi} - this is a polynomial growth in the free group case and an essential polynomial growth in the
case of a torsion free hyperbolic group with infinitely many ends). Outside the incompressible tori bounding these
submanifolds, the $G_\alpha$-geodesics behave like in a hyperbolic group
\cite{Farb, Osin}. Thus the number of intersection-points of the strip with a ball of
region $k$ growths at most linearly with $k$ outside the mapping-tori of the
$\alpha$-polynomial growth subgroups. Inside these mapping-tori, it is easily
proved that this same number is bounded above by a polynomial of degree
$3$. Indeed
on the one hand the $\alpha$-growth of an element is bounded above by a polynomial of
degree $2$. On the other hand the $G_\alpha$-geodesics remain in a bounded neighborhood of a
corridor between the orbits of their entrance- and exit-points. Thus the bound
is given by the product of a linear map with a degree $2$ polynomial (in general - free group case or torsion free hyperbolic group with infinitely many ends case -
this is the product of a linear map with a degree $d$ polynomial).
\end{proof}

\begin{proof}[Proof of Theorem \ref{resultat 1}]
We consider a $G$-tree $\mathcal T$ as given by Theorem \ref{le cas
libre}. By Propositions \ref{a faire}, \ref{verification de la
condition CS}, \ref{verification de la condition CP}, Theorem
\ref{theoreme 2.4-6.5} applies and gives the first point of Theorem
\ref{resultat 1}. Proposition \ref{Francois6.5} and Theorem
\ref{theoreme 2.4-6.5} give the Poisson boundary.
\end{proof}

\begin{corollary}
\label{corollaire pour GL} With the assumptions and notations of
Theorem \ref{resultat 1}:

\begin{enumerate}
  \item There is a topology on $G_\alpha \cup \partial G$ such
  that {\bf P}-almost every sample path ${\bf x} = \{x_n\}$ of the random walk
  converges to some $x_\infty \in \partial G$.
  \item The hitting measure $\lambda$ (i.e. the distribution of $x_\infty$) is a non-atomic
measure on $\partial G$ such that $(\partial G,\lambda)$ is a $\mu$-boundary of $(G_\alpha,\mu)$,
and this is the unique $\mu$-stationary measure on $\partial G$.
  \item If $\mu$ has finite first logarithmic moment and finite entropy with respect to some finite word-metric
on $G_\alpha$, then the measured space $(\partial G,\lambda)$ is the Poisson
boundary of $(G_\alpha,\mu)$.
\end{enumerate}
\end{corollary}

\begin{proof}
The map $q \circ \mathcal Q \colon \partial G \rightarrow
\widetilde{\mathcal T}$ (see Proposition \ref{a faire} for the definition of $q$ and Lemma \ref{motivons} for the definition of the LL-map $\mathcal Q$) is continuous, surjective and $G_\alpha$-equivariant. Item (1)
of Theorem \ref{resultat 1} then gives Item (1) of the current corollary. The map
$q \circ \mathcal Q$ is a continuous projection such that the disjoint sets
that get identified to a single point are
the disjoint translates of a finite number of disjoint compact subsets of
$\partial G$. We can thus define a hitting measure $\lambda$ on $\partial G$
 such that $(q \circ \mathcal Q)_* \lambda$ is the
$\mu$-stationary measure on $\widetilde{\mathcal T}$ given by
Theorem \ref{resultat 1} (the pre-image of a point has $\lambda$-measure $0$). The non-atomicity,
$\mu$-stationarity and unicity of $\lambda$ follow from the non-atomicity, $\mu$-stationarity
and the unicity of the measure on $\widetilde{\mathcal T}$ given by Theorem \ref{resultat 1}. We
so got Item (2), and Item (3) is a straightforward consequence of Item (3) of Theorem \ref{resultat 1}.
\end{proof}

\subsection{Generalizations to free and hyperbolic groups}

The following result is now given by Theorem \ref{le cas hyperbolique}.

\begin{theorem}

\label{resultat 3} Theorem \ref{resultat 1} and Corollary
\ref{corollaire pour GL} remain true

 \begin{itemize}
   \item either if one substitutes a
non-abelian free group to the fundamental group of a compact hyperbolic surface,
  \item or more generally if one substitutes a torsion free hyperbolic group with infinitely many
ends $G$ to the fundamental group of a compact hyperbolic surface and the automorphism
has essential exponential growth.
\end{itemize}
\end{theorem}

Observe that, although $G$ has infinitely many ends, the cyclic
extension $G \rtimes_\alpha \mz$ is one-ended.

\section{Cyclic extensions over polynomially growing automorphisms}\label{cyclicpoly}

The core of this section is to prove Theorem \ref{resultat
interessant 1bis} below. Applications to particular classes of
groups follow, see Theorem \ref{resultat 2bis}.

\begin{definition}
Let $\mathcal T$ be a $\mr$-tree. {\em Closing a bi-infinite geodesic}
in $\mathcal T$ means identifying its endpoints in $\partial \mathcal T$.
\end{definition}

\begin{theorem}
\label{resultat interessant 1bis} Let $G$ be a hyperbolic group. Let
$\alpha \in \Aut{G}$ be such that there exists a simplicial $\alpha$-invariant $G$-tree $\mathcal
T$ satisfying the following properties:

\begin{itemize}
  \item The $G$-action has quasi convex vertex stabilizers,
  trivial edge stabilizers and only one orbit of
edges.
  \item The $\alpha$-action fixes exactly one vertex $O$
  and fixes each $\mathrm{Stab}_G(O)$-orbit of directions at
  $O$.
\end{itemize}

Let $\mu$ be a probability measure on $G_\alpha := G
\rtimes_\alpha \mz$ whose support generates $G_\alpha$
as a semi-group. Then, if
$\widetilde{\mathcal T}$ is obtained from $\widehat{\mathcal T}$ by
closing each geodesic in the $G$-orbit of some bi-infinite geodesic:

\begin{enumerate}
  \item There is a topology on $G_\alpha \cup \partial \widetilde{\mathcal T}$ such
  that {\bf P}-almost every sample path ${\bf x} = \{x_n\}$ of the random walk
  converges to some $x_\infty \in \partial \widetilde{\mathcal T}$.
  \item The hitting measure $\lambda$ is a non-atomic measure on $\partial \widetilde{\mathcal T}$
such that $(\partial \widetilde{\mathcal T},\lambda)$ is a $\mu$-boundary of $(G_\alpha,\mu)$,
and this is the unique $\mu$-stationary measure on $\partial \widetilde{\mathcal T}$.
  \item If $\mu$ has finite first logarithmic moment and finite entropy with respect to
some finite word-metric on $G_\alpha$ then the measured space $(\partial \widetilde{\mathcal T},\lambda)$
is the Poisson boundary of $(G_\alpha,\mu)$.
\end{enumerate}
\end{theorem}

An important intermediate step is to first prove Theorem
\ref{resultat 1bis} below. However we need an additional notion
before its statement:

\begin{definition}

Let $\Phi \in \Out{G}$ and let $\mathcal T$ be a $\Phi$-invariant
$G$-tree. Let $\alpha, \beta \in \Aut{G}$ with $[\alpha] = [\beta] =
\Phi$. We denote by $H_\alpha$ (resp. $H_\beta$) the corresponding
isometries of $\mathcal T$, i.e. for any $w \in G$ and $P\in
\mathcal T$, $H_\alpha(wP) = \alpha(w) H_\alpha(P)$ and $H_\beta(wP)
= \beta(w) H_\beta(P)$.

The {\em $(\alpha,\beta)$-action of $G_\Phi  := G \times_\Phi
\mz$ on $\mathcal T \times \mathcal T$} is the action given by
$\Theta \colon G_\Phi \rightarrow \mathrm{Isom}({\mathcal T}
\times {\mathcal T})$ with:

$$\Theta(wt) \left\{ \begin{array}{ccccccc}
{\mathcal T} & \times & {\mathcal T} & \rightarrow & {\mathcal T} &
\times & {\mathcal T} \\
(P & , & Q) & \mapsto & (wH^{-1}_\alpha(P) & , & wH^{-1}_\beta(Q)) \\
\end{array} \right.$$

\end{definition}

\begin{theorem}
\label{resultat 1bis} Let $G$ be a hyperbolic group. Let $\alpha \in
\Aut{G}$ be such that there
exists a simplicial $\alpha$-invariant $G$-tree $\mathcal T$
satisfying the following properties:

\begin{itemize}
  \item The $G$-action has quasi convex vertex stabilizers,
  trivial edge stabilizers and only one orbit of
edges.
  \item The $\alpha$-action fixes exactly one vertex $O$
  and fixes each $\mathrm{Stab}_G(O)$-orbit of directions at
  $O$.
\end{itemize}

Let $\mu$ be a probability measure on $G_\alpha := G
\rtimes_\alpha \mz$. Then there exist $\beta \in \Aut{G}$ with
$[\alpha] = [\beta] := \Phi$ in $\Out{G}$, and a bi-infinite geodesic $A$
such that, if $\widetilde{\mathcal T}$ is obtained from
$\widehat{\mathcal T}$ by closing each geodesic in the $G$-orbit of
$A$ and $\partial \overline{G_\Phi.O}$ is the boundary  of the
closure of the $(\alpha,\beta)$-orbit of $O$
  in $\partial (\widetilde{{\mathcal T}} \times \widetilde{\mathcal T})$
then:

\begin{enumerate}
\item {\bf P}-almost every sample path ${\bf x} = \{x_n\}$
  converges to some $x_\infty \in \partial \overline{G_\Phi.O}$.
\item The hitting measure $\lambda$,
which is the distribution of $x_\infty$, is a non-atomic measure on
   $\partial \overline{G_\Phi.O}$ such that
  $(\partial \overline{G_\Phi.O},\lambda)$
  is a $\mu$-boundary of $(G_\alpha,\mu)$
  and $\lambda$ is the unique
  $\mu$-stationary probability measure on
  $\partial \overline{G_\Phi.O}$.
  \item If the measure $\mu$ has finite first
  logarithmic moment and finite entropy with respect to a word-metric
  on $G_\alpha$, then the measured space
  $(\partial \overline{G_\Phi.O},\lambda)$
  is the Poisson boundary of $(G_\alpha,\mu)$.
\end{enumerate}
\end{theorem}

\begin{remark}
Of course $G_\Phi$, $G_\alpha$ and $G_\beta$ are isomorphic groups. We adopt these different notations in the above statement to insist on the different actions: the generator of $\mz < G_\alpha$ acts as $\alpha^{-1}$ on $\widehat{\mathcal T}$, the generator of $\mz < G_\beta$ as $\beta^{-1}$ on $\widehat{\mathcal T}$ and the generator of $\mz < G_\Phi$ as $\alpha^{-1}$ on the first factor of $\widehat{\mathcal T} \times \widehat{\mathcal T}$ and as $\beta^{-1}$ on the second factor.
\end{remark}

We set $G = \langle x_1,\cdots,x_n \rangle$. The generator $t$ of
$\mz$ acts on the left as the isometry $H^{-1}$ with $H$ the
isometry of $\mathcal T$ satisfying $H(wP) = \alpha(w) H(P)$. Let us
observe that the simplicial nature of $\mathcal T$ gives us an easy
identification of $G$ with the orbit of some vertex.

\begin{lemma}
\label{greve} Choose the point $O$ given in Theorem \ref{resultat
1bis} as base-point and choose an identification of $\mathrm{Stab}_G(O)$ with the directions at $O$. If there are two orbits of vertices, choose
also a base-point $O^\prime$ in the other orbit adjacent to $O$. Then any vertex $x$
of $\mathcal T$ is associated to a unique left-class $w H_i$
($i=1,2$) as follows:

\begin{itemize}
  \item $w$ is the element of $G$ associated to the unique
  geodesic from $O$ to $x$,
  \item $H_i$ is the stabilizer either of $O$ or of $O^\prime$,
\end{itemize}
and each direction at each vertex is associated to a unique element of $G$.

\end{lemma}

Equipping ${\mathcal T}$ with the observers topology, and after
identifying $G_\alpha$ with the orbit of $O$, we get a compatible
compactification of $G_\alpha$ in a way similar to Lemma
\ref{obvious}. This compactification however does not necessarily
satisfy the (CP) property: indeed it might happen that some isometry
$vt^j$, $v \in G$, fixes more than one point. This would imply that,
for some $w \in G$, the sequences $\{(vt^j)^k\}_{k =
1,\cdots,+\infty}$ and $\{(vt^j)^kw\}_{k=1,\cdots,+\infty}$ would
not have the same limit-point.

By assumption, $t$ fixes each $\mathrm{Stab}_G(x)$-orbit of
direction at $O$. Thus, there exists an isometry of $G_\alpha$
fixing more than one vertex of $\mathcal T$, it has the form $vt$.

The important observation holds in the following well-known
observation:

\begin{lemma}
\label{cest connu} If both $vt$ and $wt^k$ fix more than one vertex
in ${\mathcal T}$ then there is $g \in G$ with $(vt)^k =g^{-1} wt^k
g$. If $vt$ and $wt^k$ fix the same edge, then $w = v \alpha^{-1}(v)
\cdots \alpha^{1-k}(v)$.
\end{lemma}

\begin{proof}
Assume that $vt$ and $wt$ both fix at least two vertices. Since
there is only one $G$-orbit of edges, without loss of generality we
can assume that $vt$ and a $G$-conjugate of $wt$, denoted by $g^{-1}
wt g$, both fix the same edge $E$. Then $vtg^{-1} t^{-1} w^{-1} g =
w \alpha^{-1}(g^{-1}) w^{-1} g$ fixes $E$. By the triviality of the
edge stabilizers, $vtg^{-1} t^{-1} w^{-1} g$ is trivial. We so get
the lemma in the case $k=1$, the generalization is straightforward.
\end{proof}

\begin{definition}

A {\em singular element} for an action of $G_\alpha$ on $\mathcal T$
is an element of $G_\alpha$ which fixes at least two vertices of
$\mathcal T$.

\end{definition}

If $vt$ is a singular element, the trick now is to consider another
action of $G_\alpha$ on ${\mathcal T}$ by making $t$ act on
${\mathcal T}$ by another automorphism in the same outer-class so
that $v t$ acts on this copy of ${\mathcal T}$ as a hyperbolic
element.

\begin{lemma}
\label{intermediaire} With the notations above: there is an
$(\alpha,\beta)$-action of $G_\Phi$ on $\mathcal T \times
\mathcal T$ such that any element of $G_\Phi$ which is a
singular element for the action induced on one of the two factors
(${\mathcal T} \times \{*\}$ or $\{*\} \times {\mathcal T}$) acts as
a hyperbolic isometry on the other factor and the axis of this hyperbolic
isometry
either is disjoint from the fixed set of the singular element, or is
an axis in this fixed set.

The axis of this hyperbolic isometry is called a {\em singular axis}.
\end{lemma}

\begin{proof}

Assume that $vt$ is a singular element. Let $T$ be the tree fixed by $vt$.
We choose a hyperbolic
element $a$ of $G$ such that $avt$ acts as an hyperbolic isometry on
$\mathcal T$, the axis of which is disjoint from $T$ if $T \neq {\mathcal T}$.
We consider the $(\alpha,\beta)$-action of $G_\Phi$ given by:

$$\Theta(wt) \left\{ \begin{array}{ccccccc}
{\mathcal T} & \times & {\mathcal T} & \rightarrow & {\mathcal T} &
\times & {\mathcal T} \\
(P & , & Q) & \mapsto & (wH^{-1}(P) & , & wv^{-1}avH^{-1}(Q)) \\
\end{array} \right.$$

If $wt$ is a singular element for the first factor, then by Lemma
\ref{cest connu}, there exists $g \in G$ with $wt = g^{-1}vtg$.
Since $vt$ acts on the second factor as the hyperbolic isometry
$avt$, $wt$ acts on this second factor as $g^{-1}avtg$, which is
also an hyperbolic isometry.

Assume now that $wt$ acts as a singular element on the second
factor. Since $vt$ acts on the second factor like $vv^{-1}avt=avt$
then $a^{-1}vt$ is a singular element for the action on the second
factor. Thus, by Lemma \ref{cest connu}, $wt$ acts on the first factor as a conjugate to
$a^{-1}vt$. Since $avt$ is a hyperbolic isometry of $\mathcal T$, so
is $a^{-1}vt$ and so is any of its conjugates.
\end{proof}

Of course any $(\alpha,\beta)$-action extends to an action on
$\widehat{\mathcal T} \times \widehat{\mathcal T}$.

\begin{definition}

A {\em singular boundary-tree} of ${\mathcal T} \times {\mathcal T}$
is a product $\partial A \times T$ or $T \times \partial A$ where:

\begin{itemize}
  \item $A$ is a singular axis.
  \item $T$ is the closure in $\mathcal T \cup \partial \mathcal T$
  of a maximal subtree of $\mathcal T$
  which is fixed by the singular element of singular axis $A$.
\end{itemize}
\end{definition}

\begin{definition}
We denote by:

\begin{itemize}
  \item $\widetilde{\mathcal T}$ the space obtained from
$\widehat{\mathcal T}$ by closing all the singular axis.
  \item $\widetilde{\mathcal T}^2$ the space obtained from
$\widetilde{\mathcal T} \times \widetilde{\mathcal T}$ by
identifying each singular boundary-tree to a point.
\end{itemize}
\end{definition}

\begin{lemma}
\label{cest la lutte finale} The space $\widetilde{\mathcal T}^2$ is
Hausdorff and compact. The $(\alpha,\beta)$-action of $G_\Phi$
on $\mathcal T \times \mathcal T$ induces an irreducible action on
$\widetilde{\mathcal T}^2$.
\end{lemma}

\begin{proof}
 The tree
$\mathcal T$ is simplicial thus separable. By Proposition
\ref{merciCHL} $\widehat{\mathcal T}$ is Hausdorff and compact. By
construction, there is an axis $A$ of $\mathcal T$ such that
$\widetilde{\mathcal T}$ is obtained from $\widehat{\mathcal T}$ by
identifying two ends in $\partial {\mathcal T}$ if and only if there
are the ends of an axis $gA$, $g \in G$. This is easily seen to be a
closed property in $\partial \mathcal T \times \partial \mathcal T$.
It readily follows that $\widetilde{\mathcal T}$ is Hausdorff and
compact. By definition of a singular axis, there is an element $w$
of $G$ such that $t . A = w.A$. It follows that the whole orbit
$G.A$ is invariant under the $G_\Phi$-action. Thus the $G_\Phi$-action induces an action on $\widetilde{\mathcal T}$. This action
is irreducible because it was on $\widehat{\mathcal T}$. It readily
follows that $\widetilde{\mathcal T} \times \widetilde{\mathcal T}$
is Hausdorff, compact and the $(\alpha,\beta)$-action of $G_\Phi$ is an irreducible action. From Lemma \ref{uneetuneseule}, to pass
to $\widetilde{\mathcal T}^2$, one identifies to a point each
singular boundary-tree in a single $G$-orbit. The arguments for
completing the proof are then the same as the arguments to pass from
$\widehat{\mathcal T}$ to $\widetilde{\mathcal T}$.
\end{proof}

Although important, the following proposition is however obvious
from Lemma \ref{cest la lutte finale}:

\begin{proposition}
\label{open doors} Let $O \in \widetilde{\mathcal T}^2$ be the point
whose coordinates are the fixed-point of $\alpha$.

If $P$ is any point in $\widetilde{\mathcal T}^2$, let ${\mathcal
N}(P)$ consist of a neighborhood ${\mathcal N}^{obs}(P)$ of $P$ in
$\widetilde{\mathcal T}^2$ equipped with the product topology,
together with all the elements $w t^k \in G_\alpha$ such that
$\Theta(w t^k)(O) \in {\mathcal N}^{obs}(P)$ and $|wt^k|
> N$ for some chosen positive $N$.

Then  $G_\alpha \cup \widetilde{\mathcal T}^2$, equipped with the above basis of
neighborhoods ${\mathcal N}(.)$, is a separable, compatible
compactification of the group $G_\alpha$.
\end{proposition}

\begin{proposition}
\label{CPbis} The compactification of $G_\alpha$ by
$\overline{G_\Phi.O} \subset \widetilde{\mathcal T}^2$
satisfies the (CP) condition.
\end{proposition}

Before proving Proposition \ref{CPbis}, we study with more details
the action of $G_\Phi$ on $\widetilde{\mathcal T}^2$. Even if we
will not need the full strength of the two lemmas below, they might be
useful to help the reader having a better grasp on what happens
here.

\begin{definition}

A {\em rectangle} in ${\mathcal T} \times {\mathcal T}$ is a product
of two geodesics.

A {\em singular rectangle} is a rectangle which is the product of a
singular axis with a geodesic.

A {\em corner} of a rectangle $R$ is a point in $\partial R \cap
(\partial {\mathcal T} \times \partial {\mathcal T})$, where
$\partial R = \overline{R} \setminus R$ in $\widehat{\mathcal T}
\times \widehat{\mathcal T}$.
\end{definition}

\begin{lemma}
\label{ca bosse} We consider the $(\alpha,\beta)$-action of
$G_\Phi$ on ${\mathcal T} \times {\mathcal T}$ given in Lemma
\ref{intermediaire}.

Let $R = g_1 \times g_2$ be a singular rectangle the stabilizer of
which is neither trivial nor cyclic. Then $g_1 = g_2$ and, up to
taking powers, there are a unique $v \in G$ and $wt \in G_\Phi$
such that $v$ admits $g_1$ as hyperbolic axis and $g_1$ is a
singular axis for the action of
  $wt$. In particular, the actions of $v$ and $wt$ commute and
  the stabilizer of $R$ is a $\mz \oplus \mz$-subgroup.

\end{lemma}

\begin{proof}
Since two distinct $G$-elements cannot share the same axis in
$\mathcal T$, there is at most one $v \in G$, up to taking powers,
admitting both $g_1$ and $g_2$ as hyperbolic axis. Since the action
of $G$ on $\mathcal T \times \mathcal T$ is just the diagonal
action, if there is one $v$ fixing $R$ then $g_1 = g_2$.

Let us now prove that there is at most one element in $G_\Phi$,
which does not belong to $G$ and admit both $g_1$ and $g_2$ as
hyperbolic axis. If there were two then, up to taking powers, we can
assume that they have the form $vt^k$ and $wt^k$. Then $vw^{-1} =
vt^k t^{-k} w^{-1}$ fixes the same axis. This is an element of $G$
and, since $\mathcal T$ is a $\alpha$-invariant tree, $v w^{-1}$ is
fixed by the automorphisms $v \alpha^k(.) v^{-1}$ and $w \alpha^k(.)
w^{-1}$. Therefore $\alpha^k(vw^{-1}) = w^{-1} v \alpha^k(vw^{-1})
v^{-1} w$. Since $G$ is hyperbolic, $vw^{-1}$ is then either a
torsion element, or trivial. Since it fixes a point in $\partial T$,
and edge stabilizers are trivial, it cannot be a torsion element.
Thus $vw^{-1} = 1_G$. We so got that, up to taking powers, there is
at most one $vt$ fixing $R$.

From the previous two paragraphs, if $R$ is a singular rectangle
whose stabilizer is neither trivial nor cyclic, then up to taking powers there are a
unique $v \in G$ and $wt \in G_\Phi$ fixing $R$. Since $t$ does
not act in the same way on the two factors, if $wt$ fixes $R$, then
one of the two axis is a singular axis for $wt$. Thus $v$ and $wt$
commute and we get the lemma.
\end{proof}

\begin{lemma}
\label{uneetuneseule} We consider the $(\alpha,\beta)$-action of
$G_\Phi$ on ${\mathcal T} \times {\mathcal T}$ given in Lemma
\ref{intermediaire}.

\begin{enumerate}

  \item \label{h1} There is at most one $G$-orbit of singular rectangles the stabilizers of which
  are neither trivial nor cyclic.
  \item Let $R$ be a rectangle with cyclic
  stabilizer $\langle w_1 t^{n_1} \rangle$. For any point $P \in \widehat{\mathcal T}
  \times \widehat{\mathcal T}$, the accumulation points of
  $\Theta(\langle w_1 t^{n_1} \rangle)(P)$ either are two opposite corners of
  $R$, or belong to a singular boundary-tree.
\end{enumerate}
\end{lemma}

\begin{proof}
Item (\ref{h1}) comes directly from Lemmas \ref{cest connu} (unicity
of the singular element) and \ref{ca bosse}. Consider now a
rectangle $R$ with cyclic stabilizer in $G$. If this cyclic
stabilizer is not generated by a singular element, we are in the
case where the orbit of $O$ accumulates on two opposite corners. If
the cyclic stabilizer is generated by a singular element, we are in
the second case.
\end{proof}

\begin{proof}[Proof of Proposition \ref{CPbis}]
We first recall that the map $\mathcal Q$ is well-defined for
the actions of $G$ on $\mathcal T$ (triviality of the edge-stabilizers) so that the action of $G$ alone is not a problem.
Consider a sequence $v_i t^{n_i}$ such that $\Theta(v_i t^{n_i}).O$
converges to some point $P$. If the points of the sequence belong to
infinitely many rectangles then this is also true for any
sequence $\Theta(v_i t^{n_i}w).O$, $w \in G$. The convergence of
$\Theta(v_i t^{n_i}).O$ to $P$ means that $\Theta(v_i t^{n_i}).O$
turns around $P$, and so does $\Theta(v_i t^{n_i}w).O$. Let us
assume that there exists $N \geq 0$ such that for all $i \geq N$,
all the $\Theta(v_i t^{n_i}).O$ belong to a same rectangle. Then
again either the points turn around $P$ and the conclusion for
$\Theta(v_i t^{n_i}w).O$ is straightforward. Or the rectangle has a
non-trivial stabilizer and the conclusion for $\Theta(v_i
t^{n_i}w).O$ is deduced from Lemma \ref{uneetuneseule} since each
singular boundary-tree has been collapsed to a point. We conclude by
noticing that, since $t$ fixes $O$, the above arguments readily
imply the (CP) condition.
\end{proof}

\begin{definition}
With the assumptions and notations of Lemma \ref{greve}:

The {\em exit-point at a vertex $b \in \mathcal T$ in the direction $D_b(P)$} is the element of $G$ associated to this direction by Lemma \ref{greve}. If $w$ is a geodesic in $\widehat{\mathcal T}$ between $P \in \widehat{\mathcal T}$ and $Q \in \widehat{\mathcal T}$, the {\em exit-points in $w$} are the exit-points at the vertices $v_i \in w$ of $\widehat{\mathcal T}$ in the directions $D_{v_i}(P)$ and $D_{v_i}(Q)$.

We identify $G$ with $G \times \{0\}$.  Let $b_1, b_2$ be any two distinct points in $\widehat{\mathcal T}$.
The {\em basic strip $BS(b_1,b_2)$} consists of all the exit-points in the unique $\widehat{\mathcal T}$-geodesic between $b_1$ and $b_2$.

\end{definition}

\begin{definition}

Let $\mathcal R$ be a set of representatives of $G_\Phi$-orbits
in $(\overline{G_\Phi. O} \times \overline{G_\Phi . O}
\setminus \Delta) / ((x,y) \sim (y,x))$ (we recall that $G_\Phi$ acts via the map $\Theta$
given in Lemma \ref{intermediaire}).

If $(b^\prime_1,c^\prime_1)$ and $(b^\prime_2,c^\prime_2)$ are two
distinct elements in $\mathcal R$, we define the {\em elementary strip
$ES((b^\prime_1,c^\prime_1),(b^\prime_2,c^\prime_2))$} as the union of
$BS(b^\prime_1,b^\prime_2)$ and $BS(c^\prime_1,c^\prime_2)$.

Let $(b_1,c_1)$ and $(b_2,c_2)$ be two distinct elements in
$\overline{G_\Phi . O}$. The {\em strip
$S(b_1,c_1,b_2,c_2)$} is the union of all the $w_i t^{k_i}
ES(b^i_1,c^i_1,b^i_2,c^i_2)$ such that $w_i t^{k_i} .
\{(b^i_1,c^i_1),(b^i_2,c^i_2)\} = \{(b_1,c_1),(b_2,c_2) \}$.

\end{definition}

\begin{proposition}
\label{verificationCSbis}

The compactification of $G_\alpha$ by $\overline{G_\Phi.O}
\subset \widetilde{\mathcal T}^2$, equipped with the collection of
strips $S(b_1,b_2)$, satisfies the (CS) condition.

\end{proposition}

\begin{proof}
By definition no basic strip is empty so that no strip is empty. The $G_\Phi$-equivariance is clear by
construction. It remains to check that the strip
$S(b_1,c_1,b_2,c_2)$ does not accumulate on a third boundary point
$b_0$.

There is of course no problem if the strip is finite. Let us thus assume that it is infinite. Then either $\{(b_1,c_1),(b_2,c_2)\}$ is stabilized by a non-trivial element of $G_\Phi$. Since singular axis have been closed and singular trees in the boundary have been collapsed to a point, this element is an element of $w \in G \lhd G_\Phi$. Then $(b_1,c_1)$ and $(b_2,c_2)$ are in $\partial \widehat{\mathcal T} \times \partial \widehat{\mathcal T}$ and $b_1=c_1$, $b_2=c_2$ are the two boundary points of the axis of the hyperbolic isometry $w$. It readily follows that $(b_1,c_1)$ and $(b_2,c_2)$ are the only accumulation-points. Or no non-trivial element of $G_\Phi$ stabilizes $\{(b_1,c_1),(b_2,c_2)\}$ and then the basic strip itself is infinite: this might only occur if at least one of the $b_i$ or $c_i$ is in $\partial \mathcal T$. Again in this case the only accumulation-points are by definition the points in the boundary of the strip.

\end{proof}

\begin{proposition}
\label{Francois6.5bis} Let ${\mathcal J}$ be a finite word gauge for
$G_\alpha$. Then the strips $S(b_1,b_2)$ given by Proposition
\ref{verificationCSbis} grow polynomially with respect to $\mathcal
J$.
\end{proposition}

\begin{proof}
Observe that everything has been done so that if $\{(b_1,c_1),(b_2,c_2)\}$ is stabilized
then the stabilizer is at most a cyclic group generated by $w \in G \lhd G_\Phi$, and in this case the strip consists of the exit-points in the $\widehat{\mathcal T}$-geodesic which connects the boundary
of the axis of this stabilizer. Thus the number of intersections of such a strip with a ball of radius $k$ growths linearly with $k$. We can thus assume that $\{(b_1,c_1),(b_2,c_2)\}$ is stabilized by no element in $G_\Phi$. Then the strip is the union of two basic strips and so is infinite if and only if a basic strip is infinite. Since infinite basic strips are composed of exit-points lined up along a geodesic of $\widehat{\mathcal T}$ the conclusion follows.
\end{proof}

\begin{proof}[Proof of Theorem \ref{resultat 1bis}]

 Proposition \ref{open doors} gives that
$\overline{G_\Phi.O} \subset \widetilde{\mathcal T}^2$ is a
separable, compatible compactification of $G_\alpha$ with an
irreductible action. Propositions \ref{CPbis} and
\ref{verificationCSbis} give Kaimanovich (CP) and (CS) properties.
By Theorem \ref{theoreme 2.4-6.5}, we so get the conclusions of the
first point of Theorem \ref{resultat 1bis} for the union of
$\partial \overline{G_\Phi.O} \subset \partial
\widetilde{\mathcal T}^2$ with the vertices of $\mathcal T \times
\mathcal T$ in $\overline{G_\Phi.O} \setminus G_\Phi.O$.
Since this set of vertices is invariant, and the measure $\lambda$
non-atomic, it has $\lambda$-measure zero, which gives to us Theorem
\ref{resultat 1bis}, Item (2). Theorem \ref{theoreme 2.4-6.5} and
Proposition \ref{Francois6.5bis} give the third point.
\end{proof}

\begin{proof}[Proof of Theorem \ref{resultat interessant 1bis}]
We call {\em singular points} the points resulting
from the collapsing of the singular boundary-trees. There is a
slight abuse of terminology in the lemma below when considering the
``projection on the first factor'' $\pi \colon \partial
\widetilde{\mathcal T}^2 \setminus \{\mbox{singular points}\}
\rightarrow \widetilde{\mathcal T}$. We mean of course the map
induced by the projection on the first factor from $\partial
(\widetilde{\mathcal T} \times \widetilde{\mathcal T})$ to
$\widetilde{\mathcal T}$, where $\partial
(\widetilde{\mathcal T} \times \widetilde{\mathcal T})$ denotes $(\partial \widetilde{\mathcal T} \times \widetilde{\mathcal T}) \cup (\widetilde{\mathcal T} \times \partial \widetilde{\mathcal T}) \cup (\partial \widetilde{\mathcal T} \times \partial \widetilde{\mathcal T})$.

\begin{lemma}
\label{il faudra le prouver} Let $\pi \colon \partial
\widetilde{\mathcal T}^2 \setminus \{\mbox{singular points}\}
\rightarrow \widetilde{\mathcal T}$ be the projection on the first
factor. Then:

\begin{enumerate}
  \item For any $x \in \partial \widetilde{\mathcal T}$ lying in the image of $\pi$,
  $\pi^{-1}(x) \cap \partial \overline{G_\Phi.O}$ consists of
  exactly one point.
  \item If $V({\mathcal T})$ denotes the set
  of vertices of $\mathcal T$, then there exist either one or two
  vertices $x_0,x_1$ of $\mathcal T$ such that $\{\pi^{-1}(x), x \in V({\mathcal T}) \}
  = G_\Phi.\pi^{-1}(x_0) \sqcup G_\Phi. \pi^{-1}(x_1)$ and all the translates
  $\gamma.\pi^{-1}(x_0)$ and $\gamma.\pi^{-1}(x_1)$,$\gamma \in G_\Phi$,
  are disjoint as soon as the translating
  element $\gamma$ is not in the stabilizer.
  \item The map $\pi$ is $G_\Phi$-equivariant.
\end{enumerate}
\end{lemma}

\begin{proof}
Items (2) and (3) are clear. Item (1) is a consequence of the fact
that the points in $\partial {\mathcal T}$ which have neither
trivial nor cyclic stabilizer are the endpoints of singular axis.
Indeed, if $v$ and $wt^k$ both fix $P \in \partial {\mathcal T}$
then $wt^k$ and $v$ commute so that $v^{-1} wt^k$ act as the
identity on an axis with terminal point $P$. Thus $P$ is not in the
image of $\pi$ (we defined $\pi$ on the complement of the set of
singular points).
\end{proof}

The set of singular points obviously has $\lambda$-measure zero,
where $\lambda$ is the Poisson measure. This is also true for the
sets $\pi^{-1}(x)$ with $x \in V({\mathcal T})$. Indeed this
follows from Item (2) of Lemma \ref{il
faudra le prouver} and the following

\begin{lemma}\cite[Lemma 2.2.2]{KM}
\label{fibre}
Let $G$ be a countable group, $\mu$ a probability measure on $G$,
and $B$ a $G$-space endowed with a $\mu$-stationary probability measure $\nu$.
Let $\pi : B \rightarrow C$ be any $G$-equivariant quotient map on a $G$-space $C$
on which $G$ acts with infinite orbits.
Then all the fibers $\pi^{-1}(x)$ have $\nu$-measure zero.
\end{lemma}

\begin{proof}
Let $X = \pi^{-1}(x)$ be a fiber and denote by $S$ the stabilizer of $X$
in $G$. Consider the function $f$ defined on $G/S$ by $f(\gamma)=\nu(gX)$
where $\gamma = gS \in G/S$. One has $\sum_{\gamma \in G/S} f(\gamma) \leq \nu(B) < +\infty$
therefore $f$ has a maximum value. On the other hand, since $\nu$
is $\mu$-stationary, the function $f$ satisfies the following mean-value
property : $\sum_{h\in G} \mu(h) f(h^{-1}\gamma) = f(\gamma)$, therefore
$f$ must be constant. Since the set $G/S$ is infinite, $f$ must be zero.
\end{proof}

Thus, by Item (1) of Theorem \ref{resultat 1bis} almost every sample path ${\bf x} = \{x_n\}$
  converges to some $x_\infty \in \partial \widetilde{\mathcal T}$ for the measure associated to $\pi_*\lambda$. Observe however that there is again a slight abuse here since $\pi_* \lambda$ is a priori only defined
  on $\pi(\partial
\widetilde{\mathcal T}^2 \setminus \{\mbox{singular points}\})$, which is only a subset of $\partial \widetilde{\mathcal T}$ when deprived of the zero measure set $V({\widetilde{\mathcal T}})$, but it suffices to extend it by declaring the complement to be of measure zero. From which precedes $\pi_*\lambda$ is non-atomic and this is the hitting measure for the topology we have on $G_\alpha \cup \partial \widetilde{\mathcal T}$.

Observe that $\partial \widetilde{\mathcal T}$ satisfies the proximality property of Lemma \ref{unique}. By Remark \ref{compacite non necessaire}
and Lemma \ref{unique} the measure $\pi_*\lambda$ is thus the unique $\mu$-stationary measure.

By Item (3) of Theorem \ref{resultat 1bis} and Lemma
\ref{il faudra le prouver}, $(\partial \widetilde{\mathcal T},\pi_* \lambda)$ is the Poisson boundary of $(G_\alpha,\mu)$.
\end{proof}

\subsection{Consequence of Theorems \ref{resultat interessant 1bis} and \ref{resultat 1bis}}

\begin{theorem}
\label{resultat 2bis} Let $G$ be a torsion free hyperbolic group
  with infinitely many ends. Let $G_\alpha = G \rtimes_\alpha \mz$
be the semi-direct product of $G$ with $\mz$ over an essentially polynomially
growing automorphism $\alpha$.

  Let $\mu$ be a probability measure on $G_\alpha$
whose support generates $G_\alpha$ as a semi-group. Then there exist a simplicial
$\alpha$-invariant $G$-tree $\mathcal T$ and a bi-infinite geodesic $A$ in $\mathcal T$
such that, if $\widetilde{\mathcal T}$ denotes the space obtained from
$\widehat{\mathcal T}$ by closing each geodesic in $G.A$, then:

\begin{itemize}

  \item There is a topology on $G_\alpha \cup \partial \widetilde{\mathcal T}$
  such that {\bf P}-almost every sample path ${\bf x} = \{x_n\}$
  admits a subsequence which converges to some $x_\infty \in \partial \widetilde{\mathcal T}$.

 \item The hitting measure $\lambda$ is non-atomic and it is
the unique $\mu$-stationary measure on
$\partial \widetilde{\mathcal T}$.

  \item If $\mu$ has a finite first moment then the measured space
$(\partial \widetilde{\mathcal T},\lambda)$ is the Poisson boundary of $(G_\alpha,\mu)$.

\end{itemize}

\end{theorem}

\begin{proof}
In the case of a direct product of a free group with $\mz$, the
construction of Section \ref{une introduction introductive} gives a
tree $\mathcal T$ as required by Theorem \ref{resultat interessant
1bis}. This last theorem gives the conclusion in this case.

In the case of a non-trivial cyclic extension over an essentially polynomially growing
automorphism, Theorem \ref{le cas hyperbolique} gives a tree
$\mathcal T$ which satisfies the properties required by Theorem
\ref{resultat interessant 1bis}, at the possible exception of the
fact that $\alpha$ fixes each $\mathrm{Stab}_G(x)$-orbit of
direction. However, the former theorem gives the finiteness of
number of $\mathrm{Stab}_G(x)$-orbits of directions at each vertex
$x$. Thus, after substituting $\alpha$ by $\alpha^k$, so passing from
$G_\alpha$ to the finite index subgroup $G_{\alpha^k}$, we get a tree
$\mathcal T$ for $G_{\alpha^k}$ as required by Theorem \ref{resultat
interessant 1bis}. This last theorem gives the conclusion
for $G_{\alpha^k}$ and the conclusion for $G_\alpha$ follows from Theorem \ref{indice fini}. 
\end{proof}

As in the proof of Corollary \ref{corollaire pour GL}, the existence of the LL-map ${\mathcal Q} \colon \partial G
\rightarrow \mathcal T$ gives:

\begin{corollary}
\label{le corollaire interessant} With the assumptions and notations
of Theorem \ref{resultat 2bis}:

\begin{enumerate}
  \item There is a topology on $G_\alpha \cup \partial G$ such that {\bf P}-almost every sample path
  admits a subsequence which converges to some $x_\infty \in \partial G$.
  \item The hitting measure $\lambda$ on $\partial G$ is non-atomic and
  this is the unique $\mu$-stationary measure $\lambda$ on $\partial G$.
  \item If $\mu$ has a finite first moment, then the measured space
$(\partial G,\lambda)$ is the Poisson boundary of $(G_\alpha,\mu)$.
\end{enumerate}
\end{corollary}

\begin{remark}
In the case of a direct product there is no need to pass to a finite-index subgroup so that
the conditions on the measure may be relaxed to ``finite first logarithmic moment and finite entropy''
in Theorem \ref{resultat 2bis} and Corollary \ref{le corollaire interessant}. Moreover there is no
need to pass to a subsequence in the first items of these results.
\end{remark}

\section{Extensions by non-cyclic groups}\label{noncyclic}

The plan of this section is parallel to the plan of the previous
one. As a main goal we have the following

\begin{theorem}
\label{dernier resultat interessant} Let $\mathcal G$ be any one of
the following two kinds of groups:

\begin{itemize}

\item A direct product $G \times \F{k}$ where $G$ is a torsion free hyperbolic group with
infinitely many ends.

\item A semi-direct product $G \rtimes_\theta \mathcal P$ of a free group
  $G = \F{n}$ with a finitely generated subgroup $\mathcal P$ over a
  monomorphism $\theta \colon {\mathcal P} \rightarrow \Out{\F{n}}$
  such that $\theta({\mathcal P})$ consists entirely of polynomially growing outer automorphisms.

\end{itemize}

Let $\mu$ be a probability measure on $\mathcal G$ whose support
generates $\mathcal G$ as a semi-group. Then there exists a finite index subgroup
$\mathcal U$ in $\mathcal P$, a simplicial $\theta(\mathcal
U)$-invariant $G$-tree ${\mathcal T}$, and a set $\mathcal A$ of
bi-infinite geodesics in $\mathcal T$ whose union forms a proper
subtree of $\mathcal T$ such that, if $\widetilde{\mathcal T}$
denotes the space obtained from $\widehat{\mathcal T}$ by closing
each axis in $G.\mathcal A$, then:

\begin{enumerate}
  \item There is a topology on ${\mathcal G} \cup \partial \widetilde{\mathcal T}$
  such that {\bf P}-almost every sample path admits a subsequence
which converges to $x_\infty \in \partial \widetilde{\mathcal T}$.
  \item The hitting measure $\lambda$ is non-atomic and this is the unique
$\mu$-stationary measure on $\partial \widetilde{\mathcal
T}$.
  \item If $\mu$ has finite first moment, then the measured space
$(\partial \widetilde{\mathcal T},\lambda)$ is the Poisson
boundary of $(\mathcal G,\mu)$.
\end{enumerate}

\end{theorem}

As in the previous section, we will first focus on an intermediate
result involving the product of two copies of a $G$-invariant tree.

\begin{definition}
 Let
$\theta_0, \theta_1 \colon {\mathcal U} \rightarrow \Aut{\F{n}}$ be
two monomorphisms. Assume that $[\theta_0(\mathcal U)] = [\theta_1(\mathcal
U)] := \theta(\mathcal U)$ in $\Out{G}$. Let $\mathcal T$ be a
$\theta(\mathcal U)$-invariant $G$-tree.

The {\em $(\theta_0,\theta_1)$-action of $\mathcal G = G
\rtimes_\theta {\mathcal U}$ on ${\mathcal T} \times {\mathcal T}$}
is the action given by $\Theta \colon \mathcal G \rightarrow
\mathrm{Isom}({\mathcal T} \times {\mathcal T})$ with:

$$\Theta(wu) \left\{ \begin{array}{ccccccc}
{\mathcal T} & \times & {\mathcal T} & \rightarrow & {\mathcal T} &
\times & {\mathcal T} \\
(P & , & Q) & \mapsto & (w H^{-1}_{\theta_0(u)}(P) & , & w H^{-1}_{\theta_1(u)}(Q)) \\
\end{array} \right.$$

\end{definition}

\begin{theorem}
\label{dernier resultat} Let $\mathcal G$ be any one of the
following two kinds of groups:

\begin{itemize}

\item A direct product $G \times \F{k}$ where $G$ is a torsion free hyperbolic group with
infinitely many ends.

\item A semi-direct product $G \rtimes_\theta \mathcal P$ of a free group
  $G = \F{n}$ with a finitely generated subgroup $\mathcal P$ over a
  monomorphism $\theta \colon {\mathcal P} \rightarrow \Out{\F{n}}$
  such that $\theta({\mathcal P})$ consists entirely of polynomially growing outer automorphisms.

\end{itemize}

Let $\mu$ be a probability measure on $\mathcal G$ whose support
generates $\mathcal G$ as a semi-group.

There exist

\begin{itemize}
  \item two monomorphisms $\theta_0, \theta_1 \colon {\mathcal P}
\rightarrow \Aut{G}$ with $[\theta_0(\mathcal P)] =
[\theta_1(\mathcal P)] = \theta(\mathcal P)$ and ${\mathcal G} = G
\rtimes_\theta {\mathcal P}$,
  \item a finite index subgroup $\mathcal U$ in $\mathcal P$,
  \item a simplicial $\theta(\mathcal U)$-invariant $G$-tree ${\mathcal T}$,
  \item a set $\mathcal A$ of bi-infinite geodesics in $\mathcal T$ whose
  union forms a proper subtree of
  $\mathcal T$
\end{itemize}

such that, if $\widetilde{\mathcal T}$ denotes the space obtained
from $\widehat{\mathcal T}$ by closing each geodesic in $G.\mathcal
A$ and $\partial \overline{{\mathcal
    G}.O}$ denotes the boundary
  of the closure of the $(\theta_0,\theta_1)$-orbit of $O$
  in $\partial (\widetilde{{\mathcal T}} \times \widetilde{\mathcal
    T})$ then:

\begin{enumerate}
\item {\bf P}-almost every sample path ${\bf x} = \{x_n\}$
  admits a subsequence which converges to some $x_\infty \in \partial \overline{{\mathcal G}.O}$.
\item The hitting measure $\lambda$ is a non-atomic measure and this is the unique
  $\mu$-stationary probability measure on $\partial \overline{{\mathcal G}.O}$.
\item If the measure $\mu$ has finite first moment with respect to a word-metric
  on $\mathcal G$ then the measured space $(\partial \overline{{\mathcal G}.O},\lambda)$
  is the Poisson boundary of $(\mathcal G,\mu)$.
\end{enumerate}
\end{theorem}

\begin{proof}
The proof follows exactly the same scheme as the proof of Theorem
\ref{resultat 1bis}. The only difference lies in the singular
elements.

Let us first consider the case where $\mathcal G$ is a direct
product $G \times \F{2}$ (there is no loss of generality in taking
$\F{2}$ instead of $\F{k}$). We first choose a non trivial $\alpha
\in \Inn{G}$ and construct a $\alpha$-invariant $G$-tree $\mathcal
T$ as in the previous section. Let $t_1,t_2$ be the two generators
of $\F{2}$ and let $x_1,\cdots,x_r$ be the generators of $G$. We
make them act on $\mathcal T$ as two elliptic isometries fixing a
unique vertex $v_i$, $v_1 \neq v_2$ and $v_1, v_2$ are the two
vertices of some edge $E$: $t_1$ acts as $\alpha$ whereas $t_2$ acts
as $x_1 \alpha x^{-1}_1$ for instance. This gives the
$\theta_0$-action. For each $t_i$ there is a unique $g_i \in G$ such
that $g_i t_i$ fixes $E$. Then we consider another copy of $\mathcal
T$ and make act the $t_i$'s in such a way that $g_i t_i$ acts as a
hyperbolic isometry $H_i$ and the subgroup $\mathcal H = \langle
H_1, H_2 \rangle$ is free. This gives the $\theta_1$-action. Now the
subtree $T \varsubsetneq \mathcal T$ announced by Theorem
\ref{dernier resultat} is the tree of this free subgroup. We
consider the collection of bi-infinite geodesics ${\mathcal
A}^\prime$ given by the collection of the axis of the elements of
$\mathcal H$ and the bi-infinite geodesics which are accumulations
of such axis. The set $\{(X,Y) \in
\partial T \times \partial T \mbox{ s.t. } \exists \mbox{   } A \in
{\mathcal A}^\prime \mbox{ with } X,Y \in \partial A \}$ is a closed
subset of $\partial \mathcal T \times \partial \mathcal T$. Thus
closing the bi-infinite geodesics in ${\mathcal A}^\prime$ yields,
as in the previous section, a Hausdorff, compact space. We now set
${\mathcal A} = G {\mathcal A}^\prime$ and we close the bi-infinite
geodesics in $\mathcal A$. The space $\widetilde{\mathcal T}$ we get
by closing the bi-infinite geodesics in $\mathcal A$ is Hausdorff
and compact by the same argument as in the previous section. The
$(\theta_0,\theta_1)$-action on $\mathcal T \times \mathcal T$ is
the action made explicit above. All the arguments until the end are
then copies of already given arguments.

The adaptation to the case of the extension of a free group by a
polynomial growth subgroup is done as follows: Theorem
\ref{Bestvina2} gives a finite-index subgroup $\mathcal U$ for which
there exists a $\theta_{|{\mathcal U}}({\mathcal U})$-invariant
$\F{n}$-tree $\mathcal T$. Assuming without loss of generality that
$\mathcal U$ is $2$-generated, there exist as before $g_i \in G$
such that $g_i t_i$ both fix the same edge $E$ since the $t_i$'s fix
each $\mathrm{Stab}_G(v)$-orbit of directions at $v$. We leave the
reader work out the details for the remaining arguments. We conclude
at the end by Theorem \ref{indice fini} which allows us to recover the
conclusion for the original group $G \rtimes_\theta \mathcal P$, of
which $G \rtimes_{\theta_{|{\mathcal U}}} \mathcal U$ is a
finite-index subgroup.
\end{proof}

\bigskip

Theorem \ref{dernier resultat interessant} is deduced from Theorem
\ref{dernier resultat} in the same way as Theorem \ref{resultat
interessant 1bis} was deduced from Theorem \ref{resultat 1bis}:
Lemma \ref{il faudra le prouver} still holds here. \hfill $\Box$ \\

As was previously done, the existence of the LL-map $\mathcal Q$
gives the

\begin{corollary}
\label{le dernier corollaire interessant} With the assumptions and
notations of Theorem \ref{dernier resultat interessant}:

\begin{enumerate}
  \item There is a topology on ${\mathcal G} \cup \partial G$
  such that {\bf P}-almost every sample path admits a subsequence which converges to $x_\infty \in \partial G$.
  \item The hitting measure $\lambda$ is non-atomic and this is the unique
$\mu$-stationary measure on $\partial G$.
  \item If $\mu$ has a finite first moment, then the measured space $(\partial G,\lambda)$ is the Poisson
boundary of $(\mathcal G,\mu)$.
\end{enumerate}
\end{corollary}

\begin{remark}
In the case where $\mathcal G$ is a direct product or is a semi-direct product over a {\em unipotent subgroup}
(see Theorem \ref{Bestvina2})
of polynomially growing automorphisms then there is no need to pass to a finite-index subgroup so that
the conditions on the measure in item $(3)$ may be relaxed to ``finite first logarithmic moment and finite entropy''
in Theorem \ref{dernier resultat interessant} and Corollary \ref{le dernier corollaire interessant}. Moreover there is
no need to pass to a subsequence in the first items of these results.
\end{remark}

\bigskip

\section{Examples}\label{Examples}

\subsection{Extension of a free group by a free group of polynomially growing automorphisms}
\label{exemple1}

Let $\F{3} = \langle a,b,c \rangle$ and let ${\mathcal U} := \F{2} =
\langle t_1,t_2 \rangle$. We define $\theta(t_1) := \alpha \in
\Aut{\F{3}}$ and $\theta(t_2) := \beta \in \Aut{\F{3}}$ by:

\begin{enumerate}
  \item $\alpha(a) = a$, $\alpha(b) = b$ and $\alpha(c) = ca$.
  \item $\beta(a) = a$, $\beta(b) = b$ and $\beta(c) = cb$.
\end{enumerate}

The morphism $\theta \colon {\mathcal U} \rightarrow \Aut{\F{n}}$ is
a monomorphism, i.e. $\langle \alpha,\beta \rangle$ is a free
subgroup of rank $2$ of $\Aut{\F{3}}$ because, for any non-trivial
reduced word $u$ in $\alpha^{\pm 1}, \beta^{\pm 1}$, the word $u(c)$
is non-trivial. Hence there are no non-trivial relation between
$\alpha^{\pm 1}$ and $\beta^{\pm 1}$.

The subgroup $\langle \alpha,\beta \rangle$ is a subgroup of
polynomially growing automorphisms. Indeed $a$ and $b$ have
$0$-growth under $\alpha$ and $\beta$, so under any $u \in {\mathcal
U}$, whereas $c$ has linear growth. Hence for any $w \in \F{3}$ the length
of $u(w)$ growths
at most linearly with respect to $|u|_{\langle a,b,c \rangle}$.

\subsection{Extension of a free group by an abelian group of polynomially growing automorphisms}
\label{exemple2}

Let $\alpha \in \Aut{\F{2}}$, $\F{2} = \langle a,b \rangle$, be
defined by $\alpha(a) = a$, $\alpha(b) = ab$. This is a polynomially
growing automorphism.

We consider the free group $\F{2k} = \langle a_1,b_1,\cdots,a_k,b_k
\rangle$ and the automorphisms $\alpha_1,\cdots,\alpha_k$ of
$\F{2k}$ defined by
$$\alpha_i(a_i) = a_i \mbox{ for any } i=1,\cdots,k$$
$$\alpha_i(b_i) = b_i a_i \mbox{ if } i=j$$
$$\alpha_i(b_j) = b_j \mbox{ if } i \neq j.$$
The aubgroup $\langle \alpha_1,\cdots,\alpha_k \rangle <
\Aut{\F{2k}}$ is a $\mz^k$-subgroup of polynomially growing
automorphisms in $\Aut{\F{2k}}$.

\section{Relative hyperbolicity and application to Poisson
boundaries}\label{rh}

The aim of this section is to give another description of the
Poisson boundary of a group $G_{\alpha} = G \rtimes_\alpha \mz$ when $G$ is
either a free group or the fundamental group of a compact hyperbolic
surface, and $\alpha$ is an exponentially growing automorphism.
There seems to be nothing really new here for geometric group
theorists, but only a compilation of folklore, well-known and more
recent results.

\subsection{Poisson boundary of strongly relatively hyperbolic
groups}

The Proposition \ref{dahmani} below describes the Poisson boundary of a strongly
relatively hyperbolic group in terms of its relative hyperbolic boundary. It can be 
easily deduced from \cite{Kaim1}. Beware however
that the hyperbolicity of the coned-off Cayley graph (see below) of Farb
\cite{Farb} is not sufficient: we really need the strong version of
the relative hyperbolicity. We briefly recall some basic facts about relative
hyperbolicity and otherwise invite the interested reader to consult \cite{Farb}, \cite{Bowditch}
for various definitions of relative hyperbolicity. If $G$ is a discrete group with generating set $S$, 
and $H$ a is a subgroup of $G$, the {\em coned-off Cayley graph $\Gamma^H_S(G)$ of $(G,H)$} is the graph 
obtained from by adding to $\Gamma_S(G)$ (the Cayley graph of $G$ with respect to $S$) a vertex $v(gH)$ 
for each left $H$-class and an edge of length $\frac{1}{2}$ between $v(gH)$ and all the vertices of 
$\Gamma_S(G)$ associated to elements in the class $gH$. The {\em weak relative hyperbolicity of $G$ 
with respect to $H$} just requires the hyperbolicity of $\Gamma^H_S(G)$. The {\em strong relative 
hyperbolicity} requires in addition that $\Gamma^H_S(G)$ satisfy the so-called {\em Bounded Coset 
Penetration} property, see \cite{Farb}. We will not recall its definition here but just say that this 
property forbids that two left $H$-classes remain parallel along arbitrarily long paths in $\Gamma_S(G)$. 
Combined with the hyperbolicity of $\Gamma^H_S(G)$, it follows that any two left $H$-classes separate 
exponentially so that, in particular, they define distinct points in the {\em relative hyperbolic boundary
$\partial^{RH}(G,{H})$}. Of course some care has to be taken in order to get a correct definition of 
this boundary, since $\Gamma^H_S(G)$ is not a proper space (closed balls are not compact) as soon as 
$H$ is infinite. We refer the reader to \cite{Bowditch} or \cite{Asli} to the definition of this relative 
hyperbolic boundary. The topology used by Bowditch is close in spirit to the observers topology used 
before in the current paper to deal with the non-properness of $\mr$-trees. Another construction of 
this relative hyperbolic boundary can be found in \cite{Groves} where the author takes more care in 
constructing a proper space associated to $(G,H)$ but then gets the relative hyperbolic boundary in 
an easier way. This last definition is closer in spirit to Gromov approach of relative hyperbolicity 
as initiated a long time ago in the seminal paper \cite{Gromov}. Of course all these notions have a 
straightforward generalization when substituting a finite family of subgroups $\mathcal H$ to a single 
subgroup $H$.

\begin{proposition}
\label{dahmani} Let $G$ be a finitely generated group which is
strongly hyperbolic relatively to a finite family of subgroups
$\mathcal H$. Let $\partial^{RH}(G,{\mathcal H})$ be the relative
hyperbolic boundary of $(G,{\mathcal H})$. Let $\mu$ be a
probability measure on $G$ whose support generates $G$ as a
semi-group. Then:

\begin{enumerate}
  \item {\bf P}-almost every sample path ${\bf x} = \{x_n\}$
  converges to some $x_\infty \in \partial^{RH}(G,{\mathcal H})$.
  \item The measure $\lambda = \pi({\bf P})$ on $\partial^{RH}(G,{\mathcal H})$ which
  is the distribution of $x_\infty$
  is $\mu$-stationary and non-atomic. It is the unique $\mu$-stationary measure on
  $\partial^{RH}(G,{\mathcal H})$. The $G$-space $(\partial^{RH}(G,{\mathcal H}),\lambda)$ is a
  $\mu$-boundary of $(G,\mu)$.
  \item If $\mu$ has finite first logarithmic moment and finite entropy then
  the measured space $(\partial^{RH}(G,{\mathcal H}),\lambda)$ is the Poisson boundary of $(G,\mu)$.

\end{enumerate}

\end{proposition}

\subsection{Applications to cyclic extensions of free and surface groups}

We need now some material borrowed from \cite{GauteroLustig}. Let
$\Phi \in \Out{\F{n}}$. A family of $\Phi$-polynomially growing
subgroups ${\mathcal H} = (H_1,\cdots,H_r)$ is called {\em
exhaustive} if every element $g \in \F{n}$ of polynomial growth is
conjugate to an element contained in some of the $H_{i}$. The family
$\mathcal H$ is called {\em minimal} if no $H_{i}$ is a subgroup of
any conjugate of some $H_{j}$ with $i \neq j$. The following
proposition is well-known among the experts of free group
automorphisms, see \cite{Levitt} or \cite{GauteroLustig} for a
proof.

\begin{proposition} \cite{Levitt,GauteroLustig}
\label{polynomialfamily}  Every outer automorphism $\Phi \in
\Out{\F{n}}$ possesses a {\em $\Phi$-characteristic family
${\mathcal H}(\Phi)$}, that is a family satisfying the following
properties:
 \smallskip

\noindent (a)  ${\mathcal H}(\Phi) = (H_1, \ldots , H_r)$ is a
finite, exhaustive, minimal family of finitely generated subgroups
$H_{i}$ that are of polynomial growth.

\smallskip
\noindent (b)  The family ${\mathcal H}(\Phi)$ is uniquely
determined, up to permuting the $H_{i}$ or replacing any $H_{i}$ by
a conjugate.

\smallskip
\noindent (c)  The family ${\mathcal H}(\Phi)$ is $\Phi$-invariant
(up to conjugacy).
\end{proposition}

We need to precise a little bit more this notion of ``invariance''
for a family of subgroups, with respect to the action of an
automorphism (and not only an outer automorphism).

For any $\alpha \in \Aut{\F{n}}$, a family of subgroups $\mathcal H
= (H_1,\ldots, H_r)$ is called {\em $\alpha$-invariant up to
conjugation} if there is a permutation $\sigma$ of $\{1, \ldots,
r\}$ as well as elements $h_{1}, \ldots, h_{r} \in G$ such that
$\alpha(H_k) = h_k H_{\sigma(k)} h^{-1}_k$ for each $k \in \{1,
\ldots, r\}$. Let ${\mathcal H} = (H_1,\ldots,H_r)$ be a finite
family of subgroups of $G$ which is $\alpha$-invariant up to
conjugacy. For each $H_{i}$ in $\mathcal H$ let $m_{i} \geq 1$ be
the smallest integer such that $\alpha^{m_{i}}(H_{i})$ is conjugate
in $G$ to $H_{i}$, and let $h_{i}$ be the conjugator:
$\alpha^{m_{i}}(H_{i}) = h_{i} H_{i} h_{i}^{-1}$ (this conjugator is well-defined because
the subgroups in $\mathcal H$ are malnormal - see Remark \ref{malnormal}). We define the {\em
induced mapping torus subgroup}:
$$
H^\alpha_{i} =\, \, \langle H_{i}, h^{-1}_i t^{m_i} \rangle \, \,
\subset \, G_{\alpha}$$

\begin{definition}
\label{inducedmappingtorus} Let ${\mathcal H} = (H_1,\ldots,H_r)$ be
a finite family of subgroups of $G$ which is $\alpha$-invariant up
to conjugacy. A family of induced mapping torus subgroups
$${\mathcal H}^\alpha =
(H^\alpha_{1}, \ldots, H^\alpha_{q})$$ as above is the {\em mapping
torus of $\mathcal H$ with respect to $\alpha$} if it contains for
each conjugacy class in $G_{\alpha}$ of any $H^\alpha_{i}$, for $i =
1, \ldots, r$, precisely one representative.
\end{definition}

The following theorem is from \cite{GauteroLustig} in the case where
$G = \F{n}$ and from \cite{Gautero3} (see also
\cite{GauteroLustigvieux}) in the surface case.

\begin{theorem}
\label{lustigmoi} The group $G \rtimes_\Phi \mz$ is strongly
hyperbolic relative to the mapping-torus of a $\Phi$-characteristic
family.
\end{theorem}

We can now state the theorem of this section, which follows directly
from Proposition \ref{dahmani} and Theorem \ref{lustigmoi}:

\begin{theorem}
\label{consequence facile} Let $G$ be the fundamental group of a
compact hyperbolic surface or the free group of rank $n$. Let $\Phi
\in \Out{G}$ be an exponentially growing outer automorphism and let
$\alpha$ be any automorphism in the class $\Phi$. Let $\mu$ be a
probability measure on $G_\alpha = G \rtimes_\alpha \mz$ whose support 
generates $G_\alpha$ as a semi-group. Then:

\begin{enumerate}
  \item {\bf P}-almost every sample path ${\bf x} = \{x_n\}$
  converges to some $x_\infty \in \partial^{RH}(G_\alpha,{\mathcal H}^\alpha(\Phi))$.
  \item The hitting measure $\lambda$, which is the distribution
  of $x_\infty$, is a non-atomic measure on
  $\partial^{RH}(G_\alpha,{\mathcal H}^\alpha(\Phi))$
  such that $(\partial^{RH}(G_\alpha,{\mathcal H}^\alpha(\Phi)),\lambda)$
  is a $\mu$-boundary of $(G_\alpha,\mu)$ and $\lambda$ is the unique
  $\mu$-stationary probability measure on $\partial^{RH}(G_\alpha,{\mathcal H}^\alpha(\Phi))$.
  \item If the measure $\mu$ has finite first logarithmic
  moment and finite entropy with respect to a word-metric
  on $G_\alpha$, then the measured space $(\partial^{RH}(G_\alpha,{\mathcal H}^\alpha(\Phi)),\lambda)$
  is the Poisson boundary of $(G_\alpha,\mu)$.
\end{enumerate}
\end{theorem}

When the outer automorphism is hyperbolic, meaning that any
conjugacy-class has exponential growth under $\Phi$, the
$\Phi$-characteristic family is trivial. This is exactly the case
where $G_\alpha$ is a hyperbolic group
\cite{BestvinaFeighn,Gautero1,Gautero3} so that, in this case,
Theorem \ref{consequence facile} tells nothing new with respect to
\cite{Kaim1}.

\medskip

\noindent {\em Acknowledgements:} The first author is glad to thank
M. Lustig for the numerous, almost uncountable, discussions they had
about $\mr$-trees and the actions of free groups on $\mr$-trees. He is also grateful to
the University Blaise Pascal (Clermont-Ferrand II) where he worked during the elaboration
of most of this work. Both authors thank Yves Guivarc'h for his permanent encouragement and
fruitful comments, and Fran\c cois Dahmani for pointing out Proposition \ref{dahmani}.

\bibliographystyle{plain}
\bibliography{biblioGM}

\end{document}